\documentclass[a4paper]{amsart}
\usepackage{hyperref}
\usepackage{amssymb,xcolor}
\usepackage{stmaryrd} 
\usepackage{amsmath}
\usepackage{mathrsfs}
\usepackage{amsthm}
\usepackage{mathtools}
\mathtoolsset{showonlyrefs}

\theoremstyle{plain}
\newtheorem{theorem}{Theorem}[section]
\theoremstyle{plain}

\theoremstyle{plain}
\newtheorem{definition}[theorem]{Definition}
\theoremstyle{plain}
\newtheorem{lemma}[theorem]{Lemma}
\theoremstyle{remark}
\newtheorem{remark}[theorem]{Remark}
\theoremstyle{plain}
\newtheorem{proposition}[theorem]{Proposition}
\numberwithin{equation}{section}
\theoremstyle{plain}

\begin{document}
\title{Band width estimates with lower spectral curvature bounds}

\author{Xiaoxiang Chai}
\address{School of Mathematics and Statistics, Central China Normal University,
  Wuhan, 430079 P.R. China}
\address{Department of Mathematics, POSTECH, 77 Cheongam-Ro, Nam-Gu, Pohang, Gyeongbuk, Korea 37673}
\email{ xxchai@kias.re.kr }

\author{Yukai Sun}
\address{School of Mathematics and Statistics, Henan University, Kaifeng 475004 P. R. China and Center for Applied Mathematics of Henan Province, Henan University, Zhengzhou 450046 P. R. China}
\email{sunyukai@henu.edu.cn}

\begin{abstract}
  In this work, we use the warped \( \mu \)-bubble method to study the consequences of
  a spectral curvature bound. In particular, with a lower spectral Ricci curvature bound and
  a lower spectral scalar curvature bound, we show that the band width of a torical band is bounded above. We also obtain some rigidity results.
  \end{abstract}


\maketitle


\section{Introduction}




A Riemannian \textit{band} $(M^{n},g)$ is a compact connected orientable smooth manifold with a metric $g$ and a nonempty boundary $\partial M$ such that  \[\partial M=\partial_{-}M\cup \partial_{+}M, \quad\partial_{-}M \neq \emptyset,\quad \partial_{+}M \neq \emptyset, \quad\partial_{-}M\cap \partial_{+}M=\emptyset,\]
where \( \partial_- M  \) and \( \partial_+ M \) are unions of boundary components.
The \textit{width} of $(M,g)$ is defined as 
\[\operatorname{width}_{g}(M)=\operatorname{dist}_{g}(\partial_{+}M, \partial_{-}M),\]
where $\operatorname{dist}_{g}$ is the distance on $M$ with respect to the metric $g$.

Let $\gamma >0, \Lambda\in \mathbb{R}$ be two constants. On the Riemannian band $(M^{n},g)$, if there is a positive smooth
function $u$ on \( M \backslash \partial M \) with
\begin{equation}
    -\gamma \Delta_{g}u+\tfrac{1}{2}\operatorname{Sc}_{g}u \geq \Lambda u,\quad u|_{\partial M}=0\label{eqn: scalar-spectral}, 
\end{equation}
 where $\operatorname{Sc}_{g}$ is the scalar curvature of $g$,
then we say that \( (M,g) \) satisfies a lower \textit{spectral scalar curvature bound}. Similarly, if
\begin{equation}
  -\gamma \Delta_{g}u+2\operatorname{Ric}(x)u \geq \Lambda u,\quad u|_{\partial M}=0, \label{eqn: RIc-spectral}
\end{equation}
where
\begin{equation}
  \ensuremath{\operatorname{Ric}(x)}:= \inf_{e \in T_{x}M, e \neq 0}
  \ensuremath{\operatorname{Ric}} (e,e) |e|^{- 2} \label{eq def Ric} ,
\end{equation}
then we say that \( (M,g) \) satisfies a lower \textit{spectral Ricci curvature bound}.
We use the convention that if \( \operatorname{Ric}  \) is written without an argument or with only one argument, then it means \eqref{eq def Ric}, otherwise, it means the Ricci curvature tensor.

The spectral curvature bounds \eqref{eqn: scalar-spectral} and \eqref{eqn: RIc-spectral} are easily seen to be weaker than their pointwise counterparts. The effects of a positive scalar curvature bound on a band of the form 
$M^{n}=[-1,1]\times T^{n-1}$ with a metric $g$ were studied by Gromov \cite{G2018}. 
 Here, we denote by \( T^n \) the \( n \)-dimensional torus.
 The Riemannian band $([-1,1]\times T^{n-1}, g)$  is also called a \textit{torical band}, and we call $((-1,1)\times T^{n-1}, g)$ an \textit{open torical band}.
For the torical band, Gromov proved that if the scalar curvature is bounded below by $n(n-1)$, then the band width is bounded above by $\frac{2\pi}{n}$ for $2\leq n\leq 7$. Zhu \cite{zhu2021} proved that the width of a 3-dimensional torical band is less than or equal to $\frac{\pi}{2}$ 
when the Ricci curvature is bounded below by $2$, in particular, when the sectional curvature is bounded below by $1$.
There have been other interesting developments, such as \cite{hirsch-rigid-2025}, \cite{cecchini-scalar-2024}, and \cite{rade-scalar-2023}.


The spectral curvature bound, while weaker,
has found important applications in the stable Bernstein theorem \cite{chodosh-stable-2024-1, chodosh-stable-2024, mazet-stable-2024}; see also the works of Gilles-Christian \cite{GC2021}, Antonelli-Xu \cite{AX2024}, Antonelli-Pozetta-Xu \cite{antonelli-sharp-2024}, Hong-Wang \cite{hong-splitting-2025} and \cite{he-positive-2025} for various geometric results related to the spectral curvature bound.

A natural question is what effects \eqref{eqn: scalar-spectral} and \eqref{eqn: RIc-spectral} have on the band width of a torical band. This question was raised by Gromov \cite[Section 6.1.2]{Gromov2023} (vaguely, in Item 1), and via several methods, Hirsch-Kazaras-Khuri-Zhang \cite{hirsch-spectral-2024} proved a band width estimate of a torical band under a lower spectral scalar curvature bound.

In this work, we establish some band width estimates for torical bands using the warped \( \mu \)-bubble methods under the spectral Ricci curvature bound and the spectral scalar curvature bound. 
We also have a generalization of \cite{zhu2021} to the cases of negative and zero Ricci curvature bounds.
It is not the main theme of our article, so we put this result in Appendix \ref{sec ricci}.

\begin{theorem}\label{Thm-Ricci}
    For a Riemannian band $M^{3}=[-1,1]\times T^{2}$ with the metric $g$, let $u$ be a positive smooth function on $M\setminus \partial M$ with $u=0$ on $\partial M$ such that
    \begin{equation}
      -\gamma \Delta_{g} u+2\ensuremath{\operatorname{Ric}}u \geq  \Lambda u , \label{eq spectral ricci}
    \end{equation}
    where $\Lambda>0$, $0< \gamma<4$.  Then 
    \[\operatorname{width}_{g}(M)\leq \frac{2\pi}{\sqrt{\Lambda(4-\gamma)}}.\]
\end{theorem}

By taking \( \gamma =0 \) and \( \Lambda = 4 \), we see that this band width estimate generalizes the band width estimate by Zhu \cite{zhu2021}.
Our proof strategy is by warped \( \mu \)-bubble methods. However,
the equality case of this band width estimate seems difficult to characterize by warped \( \mu \)-bubble methods.
Nonetheless, we have the following Ricci curvature rigidity.

\begin{theorem}\label{ricci rigid reformulation}
  Let \( (M = (-1,1)\times T^2,g) \) be an open torical band, and $\zeta$ be the coordinate representing the interval $(-1,1)$. Let \( \Lambda>0 \), \( 0< \gamma \leq 2\) and \( u \) satisfy \eqref{eq spectral ricci} and \( u \) tends to zero as \( \zeta \to \pm 1 \).
 If further $|\nabla \zeta| \leq 1/\ell_0$ with $\ell_0 = \pi/\sqrt{\Lambda(4-\gamma)}$,
     then 
     $(M^{3},g)$ is isometric to the band $((-\ell_0,\ell_0)\times T^{2}, g_{0})$ with
     \[g_{0}=dt^2+\phi_{1}(t)^2ds_{1}^2+\phi_{2}(t)^2ds_{2}^2, \text{ } t \in (-\ell_0, \ell_0)
     \]
     where \( s_1 \), \( s_2 \) are arc-length parameters of the circles, and
     \( \phi_{1} \), \( \phi_2 \) are positive functions given by 
\begin{align}
    \phi_{1}(t)&=\phi_{1}(0)\left(\cos(\sqrt{\Lambda(1-\tfrac{\gamma}{4})}t)\right)^{\tfrac{2-\gamma}{4-\gamma}} \exp \left ({\tfrac{\phi_{1}'(0)}{\phi_{1}(0)}\int_{0}^{t} \left(\cos(\sqrt{\Lambda(1-\tfrac{\gamma}{4})}s)\right)^{-\tfrac{1-\gamma/2}{1-\gamma/4}}ds}\right),\\
\phi_{2}(t)&=\phi_{2}(0)\left(\cos(\sqrt{\Lambda(1-\tfrac{\gamma}{4})}t)\right)^{\tfrac{2-\gamma}{4-\gamma}}\exp \left ({\tfrac{\phi_{2}'(0)}{\phi_{2}(0)}\int_{0}^{t} \left(\cos(\sqrt{\Lambda(1-\tfrac{\gamma}{4})}s)\right)^{-\tfrac{1-\gamma/2}{1-\gamma/4}}ds}\right ),
\end{align}
with \( \phi_{1}'(0) \phi_{1}(0)^{-1} = -  \phi_{2}'(0) \phi_{2}(0)^{-1}  \) and 
\( \tfrac{1}{2} (1 - \tfrac{\gamma}{2}) \Lambda \geq  2 (\tfrac{\phi_{1}'}{\phi_{1}} (0))^2  \).
\end{theorem}
\begin{remark}
  It is interesting to note that when $2<\gamma < 4$, there is no band that achieves the optimal band width.
  \end{remark}

Now we state an estimate of the band width when the Ricci curvature is replaced by the scalar curvature.

\begin{theorem}\label{Thm-scalar}
    For a Riemannian band $M^{n}=[-1,1]\times T^{n-1}$ ($3\leq n\leq 7$) with the metric $g$, let $u$ be a positive smooth function on $M\setminus \partial M$ with $u=0$ on $\partial M$ such that
    \begin{equation} \label{eq spectral scalar}
      -\gamma \Delta_{g} u+\tfrac{1}{2}\operatorname{Sc}_{g}u \geq \Lambda u,
    \end{equation}
    where $\Lambda>0, 0< \gamma<\frac{2n}{n-1}$.  Then 
    \[\operatorname{width}_{g}(M)\leq \frac{\pi}{\sqrt{\frac{- n \gamma  + \gamma  + 2 n}{4(n-1)+2\gamma(2-n)}\Lambda}}.\]
\end{theorem}

As mentioned earlier, the above band width estimate was proven by Hirsch-Kazaras-Khuri-Zhang \cite{hirsch-spectral-2024} via three different methods.
In particular, in three dimensions, they showed full rigidity (see \cite[Theorem 1.1]{hirsch-spectral-2024}) via spacetime harmonic functions.
And, in dimensions \( 3 \le n \le 7 \) they showed for the case \( \gamma= 1\) (see \cite[Theorem 1.3]{hirsch-spectral-2024}) using the warped \( \mu \)-bubble method. Here, we use the \( \mu \)-bubble method as well, but with a wider range of \( \gamma \), moreover, we have the following scalar curvature rigidity.

\begin{theorem}\label{thm rigid scalar}
  For an open Riemannian band $M^{n}=(-1,1)\times T^{n-1}$ (\( 3\leq n\leq 7 \))
  with the metric $g$, let $\zeta$ represents the coordinate of the interval $(-1,1)$. Let $u$ be a positive smooth function on $M$ such that
    \[-\gamma \Delta_{g} u+\tfrac{1}{2}\operatorname{Sc}_{g}u\geq\Lambda u,\]
    and \( u \to 0 \) as \( \zeta \to \pm 1 \). Here
    $\Lambda>0, 0< \gamma<\frac{2n}{n-1}$. We define
    \[ \ell_{1} = \frac{\pi}{2\sqrt{\frac{- n \gamma  + \gamma  + 2 n}{4(n-1)+2\gamma(2-n)}\Lambda}}\]
    to be half of the band width in Theorem \ref{Thm-scalar}.
    If further $|\nabla \zeta| \leq 1/\ell_{1}$, 
 then the band $(M,g)$ is isometric to $(-\ell_{1},\ell_{1})\times T^{n-1}$ with the metric $g=dt^2+\xi^2(t)g_{T^{n-1}}$
 where \( t \in (-\ell_1, \ell_1) \),
 \[\xi (t) = \left(\cos (\sqrt{\tfrac{\Lambda (- n \gamma + \gamma + 2 n)}{4 (n -
   1) + 2 \gamma (2 - n)} }\; t)\right)^{ 2 (2 - \gamma) (- n\gamma + \gamma + 2 n)^{- 1}}, \]
and \( g_{T^{n-1}} \) is a flat metric on the torus \( T^{n-1} \).
\end{theorem}


\

The article is organized as follows:

In Section \ref{sec basic}, we introduce our main technical tool warped \( \mu \)-bubble and recall its basics.
In Section \ref{sec spectral ricci}, we show Theorems \ref{Thm-Ricci} and \ref{ricci rigid reformulation}.
The determination of the metric in Theorem \ref{ricci rigid reformulation} is deferred to Appendix \ref{subsec-determine-the-metric}.
In Section \ref{sec spectral scalar}, we show Theorems \ref{Thm-scalar} and \ref{thm rigid scalar}.
In Appendix \ref{sec ricci}, we prove a band width estimate under the pointwise Ricci curvature bound for a 3-dimensional torical band.

\

{\textbf{Acknowledgment.}} X.C. has been partially supported by the National
Research Foundation of Korea (NRF) grant funded by the Korea government (MSIT)
(No. RS-2024-00337418) and an NRF grant No. 2022R1C1C1013511. Y.S. has been partially funded by the National Key R\&D Program of China Grant 2020YFA0712800.

\

\section{Basics of warped \texorpdfstring{$\mu$}{mu}-bubble} \label{sec basic}

In this section, we collect the basics of the warped \( \mu \)-bubble including the first and second variations, and the existence theorem.

\subsection{First and second variations of warped \texorpdfstring{$\mu$}{mu}-bubble}

On the Riemannian band $(M^{n},g)$, let $h$ be a smooth function on the interior of $M$ or on $M$. We denote the interior of $M$ by $\overset{\circ}{M}$. Choose a Caccioppoli set ${\Omega}_{0}\subset M$ with boundary $\partial {\Omega}_{0}\subset \overset{\circ}{M}$ and $\partial_{-}M\subset {\Omega}_{0}$. Let $\partial^{\ast} \Omega$ denote the reduced boundary of the Caccioppoli set $\Omega$ and $u$ be a positive smooth function on $M$. We consider the following functional 
\begin{equation}\label{eqn-omega}
   E (\Omega) = \int_{\partial^{\ast} \Omega} u^{\gamma}d\mathcal{H}^{n-1} - \int_{\Omega} (\chi_{\Omega}
   - \chi_{{\Omega}_{0}}) h u^{\gamma}d\mathcal{H}^{n} 
\end{equation}
for $\Omega\in \mathcal{C}$, where $\mathcal{C}$ is defined as
\[\mathcal{C}=\{\Omega:\text{ all Caccioppoli sets } \Omega \subset M \text{ and }\Omega\triangle {\Omega}_{0}\Subset \overset{\circ}{M}\}, \]
and $\mathcal{H}^{n}$ denotes $n$-dimensional Hausdorff measure. We usually omit the measure $\mathcal{H}^{n}$ in the integral. Then the warped $\mu$-bubbles with the warped function $u^{\gamma}$ are critical points of the functional $E(\Omega)$, which when $u\equiv 1$ or $\gamma=0$, becomes the usual $\mu$-bubble. We also call it a warped \( h\)-bubble when an emphasis on \( h \) is needed (see {\cite[Section 1.6.4]{Gromov2023}}). When we are in the settings of Theorems \ref{ricci rigid reformulation} and \ref{thm rigid scalar}, we just set \( \Omega_0 = \zeta ^{-1} ([-\tfrac{1}{2},\tfrac{1}{2}]) \). 

\begin{lemma}[First variation of warped $\mu$-bubbles]\label{lem-first-variation}
Let $\Omega_{t}$ be a smooth $1$-parameter family of regions in $\mathcal{C}$ with $\Omega_{0}=\Omega$, then $\partial\Omega_t$ are variations of $\partial \Omega$ given by some normal speed $\phi\in C^{\infty}(\partial \Omega)$ and
\begin{equation} \label{eqn-first-variation}
    \tfrac{\mathrm{d}}{\mathrm{d} t} E (\Omega_t) |_{t = 0} =
   \int_{\partial \Omega} (H + \gamma u^{- 1} u_{\nu} - h) u^{\gamma} \phi ,
\end{equation}
 where $H$ is the mean curvature of $\partial \Omega$ and $\nu$ is the outwards pointing unit normal vector on $\partial \Omega$, $u_{\nu}=\nabla^{g}_{\nu}u$. In particular, a warped $\mu$-bubble $\Omega$ satisfies
  \[H=-\gamma u^{- 1} u_{\nu} + h,\]
  along the boundary $\partial \Omega$.
\end{lemma}

\begin{lemma}[Second variation of warped $\mu$-bubbles]\label{lem-second-variation}
Consider a warped $\mu$-bubble $\Omega$ with $\partial\Omega=\Sigma$. Assume that $\Omega_{t}$ is a smooth $1$-parameter family of smooth regions in $\mathcal{C}$ with $\Omega_{0}=\Omega$, then $\partial\Omega_t$ are variations of $\partial \Omega$ given by some normal speed $\phi\in C^{\infty}(\partial \Omega)$ and
\begin{align}
  \tfrac{\mathrm{d}^2}{\mathrm{d} t^2} E (\Omega_t) |_{t = 0}
= & \int_{\partial \Omega} [- \Delta_{\partial \Omega} \phi - |A|^2 \phi
-\ensuremath{\operatorname{Ric}}(\nu, \nu) \phi - \gamma u^{- 2} u_{\nu}^2
\phi \\
& \qquad + \gamma u^{- 1} \phi (\Delta u - \Delta_{\partial \Omega} u - H
u_{\nu}) - \gamma u^{- 1} \langle \nabla_{\partial \Omega} u,
\nabla_{\partial \Omega} \phi \rangle - h_{\nu} \phi] u^{\gamma} \phi .
\label{weighted stability}
\end{align}
\end{lemma}

\begin{remark}\label{rk stable cmp minimiser}
  If $\tfrac{\mathrm{d}^2}{\mathrm{d} t^2} E (\Omega_t) |_{t = 0}\geq 0$, we say that the warped $\mu$-bubble \( \Omega \) is \textit{stable}.
  And we also call \( \partial \Omega \) a stable hypersurface of prescribed mean curvature \( - \gamma u^{-1} u_{\nu} +h \).
  Note that a stable warped \( \mu\)-bubble \( \Omega \) is not necessarily a minimiser to the functional \( E \). For convenience, we call \( \partial \Omega \) a \textit{minimising boundary} of prescribed mean curvature \( -\gamma u^{-1} u_{\nu} + h \)
  if \( \Omega \) is a minimiser to \( E \). We often omit the references to \( u \) and \( h \) if the underlying \( u \) and \( h \) are clear from the context. 
  \end{remark}

\begin{lemma}\label{lm second var rewrite}
  The second variation for the functional \( E(\Omega) \) in Lemma \ref{lem-second-variation}, that is, \eqref{weighted stability}
   can be rewritten as
    \begin{align}
\tfrac{\mathrm{d}^2}{\mathrm{d} t^2} E (\Omega_t) |_{t = 0}= & \int_{\partial \Omega} | \nabla_{\partial \Omega} \psi |^2 +
\int_{\partial \Omega} [\gamma \psi \langle \nabla_{\partial \Omega} w,
\nabla_{\partial \Omega} \psi \rangle + (\tfrac{\gamma^2}{4} - \gamma)
\psi^2 | \nabla_{\partial \Omega} w|^2] \\
& \quad + \int_{\partial \Omega} [\gamma u^{- 1} \Delta u - (|A|^2
+\ensuremath{\operatorname{Ric}}(\nu,\nu))] \psi^2 \\
& \quad - \int_{\partial \Omega} [\gamma H w_{\nu} + h_{\nu} + \gamma
w_{\nu}^2] \psi^2 . 
\end{align}
where $w = \log u$ and \( \psi \) is any smooth function on \( \partial \Omega \).
\end{lemma}
\begin{proof}
We do this by setting $\psi =\phi u^{ \gamma / 2}$ in Lemma \ref{lem-second-variation}.
We collect the following three terms
\[ \int_{\partial \Omega} [\Delta_{\partial \Omega} \phi + \gamma u^{- 1} \phi
   \Delta_{\partial \Omega} u + \gamma u^{- 1} \langle \nabla_{\partial
   \Omega} u, \nabla_{\partial \Omega} \phi \rangle] u^{\gamma} \phi, \]
and now let $\phi = u^{- \gamma / 2} \psi$, we have
\begin{align}
& \int_{\partial \Omega} [\Delta_{\partial \Omega} \phi + \gamma u^{- 1}
\phi \Delta_{\partial \Omega} u + \gamma u^{- 1} \langle \nabla_{\partial
\Omega} u, \nabla_{\partial \Omega} \phi \rangle] u^{\gamma} \phi
\\
= & \int_{\partial \Omega} [\psi \Delta_{\partial \Omega} u^{- \gamma / 2} +
2 \langle \nabla_{\partial \Omega} u^{- \gamma / 2}, \nabla_{\partial
\Omega} \psi \rangle + u^{- \gamma / 2} \Delta_{\partial \Omega} \psi +
\gamma u^{- 1 - \gamma / 2} \psi \Delta_{\partial \Omega} u \\
& \qquad + \gamma u^{- 1} \langle \nabla_{\partial \Omega} u,
\nabla_{\partial \Omega} (u^{- \gamma / 2} \psi) \rangle] u^{\gamma / 2}
\psi \\
= & \int_{\partial \Omega} \psi \Delta_{\partial \Omega} \psi \\
& \quad + \int_{\partial \Omega} [u^{\gamma / 2} \psi^2 \Delta_{\partial
\Omega} u^{- \gamma / 2} + 2 u^{\gamma / 2} \psi \langle \nabla_{\partial
\Omega} u^{- \gamma / 2}, \nabla_{\partial \Omega} \psi \rangle \\
& \hspace{3em} + \gamma u^{- 1} \psi^2 \Delta_{\partial \Omega} u + \gamma
u^{\gamma / 2 - 1} \psi \langle \nabla_{\partial \Omega} u, \nabla_{\partial
\Omega} (u^{- \gamma / 2} \psi) \rangle] \\
= & - \int_{\partial \Omega} | \nabla_{\partial \Omega} \psi |^2 \\
& \quad + \int_{\partial \Omega} [- \langle \nabla_{\partial \Omega}
(u^{\gamma / 2} \psi^2), \nabla_{\partial \Omega} u^{- \gamma / 2} \rangle +
2 u^{\gamma / 2} \psi \langle \nabla_{\partial \Omega} u^{- \gamma / 2},
\nabla_{\partial \Omega} \psi \rangle \\
& \hspace{3em} - \gamma \langle \nabla_{\partial \Omega} (u^{- 1} \psi^2),
\nabla_{\partial \Omega} u \rangle + \gamma u^{\gamma / 2 - 1} \psi \langle
\nabla_{\partial \Omega} u, \nabla_{\partial \Omega} (u^{- \gamma / 2} \psi)
\rangle] .
\end{align}
In the last line, we have used integration by parts on the two terms containing
$\Delta_{\partial \Omega} u$ and the term containing $\Delta_{\partial \Omega}
\psi$. By a direct calculation, we conclude that
\begin{align}
& \int_{\partial \Omega} [\Delta_{\partial \Omega} \phi + \gamma u^{- 1}
\phi \Delta_{\partial \Omega} u + \gamma u^{- 1} \langle \nabla_{\partial
\Omega} u, \nabla_{\partial \Omega} \phi \rangle] u^{\gamma} \phi
\\
= & - \int_{\partial \Omega} | \nabla_{\partial \Omega} \psi |^2 +
\int_{\partial \Omega} [- \gamma \psi \langle \nabla_{\partial \Omega} w,
\nabla_{\partial \Omega} \psi \rangle + (\gamma - \tfrac{\gamma^2}{4})
\psi^2 | \nabla_{\partial \Omega} w|^2]
\end{align}
 in \eqref{weighted stability}. The remaining calculation is direct.
\end{proof}

\subsection{Existence theorem}

For $u\equiv 1$, from \cite[Proposition 2.1]{zhu2021} and \cite[Section 5.1]{Gromov2023}, we have the following existence result of the $\mu$-bubble.
\begin{lemma}[Existence of $\mu$-bubble]\label{lm existence mu bubble}
For a Riemannian band $(M^{n},g)$ with $3\leq n\leq 7$, if either $h\in C^{\infty}(\overset{\circ}{M})$ with $h\to \pm\infty$ on $\partial_{\mp}M$, or $h\in C^{\infty}(M)$ with
\[h|_{\partial_{-}M}>H_{\partial_{-}M},\quad h|_{\partial_{+}M}<H_{\partial_{+}M}\]
where $H_{\partial_{-}M}$ is the mean curvature of $\partial_{-}M$ with respect to the inward normal and $H_{\partial_{+}M}$
is the mean curvature of $\partial_{+}M$ with respect to the outward normal. Then there exists an $\Omega\in \mathcal{C}$ with smooth boundary such that
\[E(\Omega)=\inf_{\Omega'\in\mathcal{C}}E(\Omega').\]
\end{lemma}

For the warped $\mu$-bubble with $u\geq\delta>0$ in \eqref{eqn-omega}, from \cite[Proposition 12]{CL2024}, we have
\begin{lemma}[Existence of warped $\mu$-bubble]\label{lem-warped-bubble}
For a Riemannian band $(M^{n},g)$ with $3\leq n\leq 7$, if $h\in C^{\infty}(\overset{\circ}{M})$ in the functional $E(\Omega)$ satisfies $h\to \pm\infty$ on $\partial_{\mp}M$, then there exists an $\Omega\in \mathcal{C}$ with smooth boundary such that
\[E(\Omega)=\inf_{\Omega'\in\mathcal{C}}E(\Omega').\]
\end{lemma}


\section{Band width estimate with the spectral Ricci curvature condition} \label{sec spectral ricci}

In this section, we prove the band width estimate with a spectral Ricci curvature bound (Theorem \ref{Thm-Ricci})
and the Ricci curvature rigidity Theorem \ref{ricci rigid reformulation}.
The proof of Theorem \ref{ricci rigid reformulation} is quite involved, and we divide the proof into several subsections. 

\subsection{Proof of the band width estimate}

We first prove Theorem \ref{Thm-Ricci}.

\begin{proof}[Proof of Theorem \ref{Thm-Ricci}]
  We prove by contradiction. We assume that
  \[\operatorname{width}_{g}(M)> 2\ell_{0}:=
 \frac{2\pi}{\sqrt{\Lambda(4-\gamma)}}
    ,\]
  then there exists a small $\delta>0$ such that
  the band
  \[ M\setminus\{x\in M: \text{ }\operatorname{dist}_{g}(x,\partial M)<\delta\}\]
  still has length greater than \( 2 \ell_0 \) and
  we can perturb it into a smooth band \( \tilde M \)
  so that the band width of \(  \tilde M \) is greater than \( 2\ell_{0} \). 
  By \cite[Lemma 7.2]{cecchini-scalar-2024}, there exists a smooth function $\chi:\tilde{M}\to [-\ell_0,\ell_0]$
on $\tilde{M}$ such that the Lipschitz constant $\operatorname{Lip}\chi<1$, $\chi^{-1}(-\ell_{0})=\partial_{-}\tilde{M}$ and $\chi^{-1}(\ell_{0})=\partial_{+}\tilde{M}$. Then the function  
\[h:=-
  \frac{2\sqrt{\Lambda}}{\sqrt{4-\gamma}}\tan\left(\frac{1}{2}\sqrt{(4-\gamma)\Lambda}\;\chi\right)=:\eta \circ \chi\]
on $\tilde{M}\setminus \partial \tilde{M}$ tends to $-\infty$ as $x\to \partial_{+}\tilde{M}$ and to $\infty$ as $x\to \partial_{-}\tilde{M}$.
We also have $u$ bounded below by a positive number on $\tilde{M}$.
It is easy to check that \( \eta: (-\ell_0,\ell_0)\) satisfies the ODE
\begin{equation} \label{eq ode spec ric}
  (1-\tfrac{1}{4}\gamma)\eta^2 + \eta' + \Lambda =0, \eta'<0.
  \end{equation}
  Using Lemma \ref{lem-warped-bubble}, we have a minimising warped $\mu$-bubble $\Omega$ such that $H=-\gamma u^{- 1} u_{\nu} + h$ along \( \partial \Omega \) and the
  second variation
$\tfrac{\mathrm{d}^2}{\mathrm{d} t^2} E (\Omega_t) |_{t = 0}$ is non-negative for any smooth variations of $\Omega$ such that $\Omega_0 = \Omega$.

Given any function $\psi \in C^{\infty}(\partial \Omega)$, we consider the variations $ \Omega_t$ of $\Omega$ such that normal speed of the variations $\partial \Omega_{t}$ is $\phi=u^{\gamma/2}\psi$ and
we can firstly rewrite the second variation as in
Lemma \ref{lm second var rewrite} and we now rewrite the second variation further based on Lemma \ref{lm second var rewrite}.
Recall \cite[(5.2)]{zhu2021}: let $\{e_1, e_2 \}$ be a local orthonormal frame of
$\partial \Omega$, then
\begin{equation}
  \ensuremath{\operatorname{Ric}} (\nu,\nu) + |A|^2 =\ensuremath{\operatorname{Ric}}
  (e_1,e_{1}) +\ensuremath{\operatorname{Ric}} (e_2,e_{2}) + H^2 - \operatorname{Sc}_{\partial \Omega} .
  \label{eq:zhu 5.2}
\end{equation}
 Since
\begin{equation}
  \gamma u^{- 1} \Delta u -\ensuremath{\operatorname{Ric}} (e_1,e_{1})
  -\ensuremath{\operatorname{Ric}} (e_2,e_2) \leq \gamma u^{- 1} \Delta u -
  2\ensuremath{\operatorname{Ric}} \leq - \Lambda \label{eq:apply lowest
  eigen}
\end{equation}
by \eqref{eq def Ric} and \eqref{eq spectral ricci},
we have that
\begin{align}
  & \int_{\partial \Omega} [\gamma u^{- 1} \Delta u - (|A|^2
  +\ensuremath{\operatorname{Ric}}(\nu,\nu))] \psi^2 \\
  = & \int_{\partial \Omega} [\gamma u^{- 1} \Delta u -
  (\ensuremath{\operatorname{Ric}}(e_1,e_{1}) +\ensuremath{\operatorname{Ric}}(e_2,e_{2})
  + H^2 - \operatorname{Sc}_{\partial \Omega})] \psi^2 \\
  = & \int_{\partial \Omega} (\operatorname{Sc}_{\partial \Omega} - H^2) \psi^2 +
  \int_{\partial \Omega} [\gamma u^{- 1} \Delta u
  -\ensuremath{\operatorname{Ric}}(e_1,e_{1}) -\ensuremath{\operatorname{Ric}}(e_2,e_{2})]
  \psi^2 \\
  \leq & \int_{\partial \Omega} (\operatorname{Sc}_{\partial \Omega} - H^2 - 
  \Lambda) \psi^2 . 
\end{align}

Using the above and that $H = - \gamma w_{\nu} + h$ (recall that $w=\log u$) in Lemma \ref{lm second var rewrite}, we arrive
\begin{align}
0 \leq & \int_{\partial \Omega} | \nabla_{\partial \Omega} \psi |^2 +
\int_{\partial \Omega} [\gamma \psi \langle \nabla_{\partial \Omega} w,
\nabla_{\partial \Omega} \psi \rangle + (\tfrac{\gamma^2}{4} - \gamma)
\psi^2 | \nabla_{\partial \Omega} w|^2] \\
& \quad + \int_{\partial \Omega} [\operatorname{Sc}_{\partial \Omega} - (- \gamma w_{\nu} +
h)^2 -  \Lambda] \psi^2 \\
& \quad - \int_{\partial \Omega} [\gamma (- \gamma w_{\nu} + h) w_{\nu} +
h_{\nu} + \gamma w_{\nu}^2] \psi^2 \\
= & \int_{\partial \Omega} | \nabla_{\partial \Omega} \psi |^2 + \operatorname{Sc}_{\partial
\Omega} \psi^2 + \int_{\partial \Omega} [\gamma \psi \langle
\nabla_{\partial \Omega} w, \nabla_{\partial \Omega} \psi \rangle +
(\tfrac{\gamma^2}{4} - \gamma) \psi^2 | \nabla_{\partial \Omega} w|^2]
\\
& \quad - \int_{\partial \Omega} [\gamma w_{\nu}^2 - \gamma h w_{\nu} + h^2
- | \nabla h| +  \Lambda] \psi^2 .
\end{align}
By Cauchy-Schwarz inequality,
\begin{equation}
  \int_{\partial \Omega} [\gamma \psi \langle \nabla_{\partial \Omega} w,
   \nabla_{\partial \Omega} \psi \rangle + (\tfrac{\gamma^2}{4} - \gamma)
   \psi^2 | \nabla_{\partial \Omega} w|^2] \leq \tfrac{1}{4} \gamma (1 -
   \tfrac{\gamma}{4})^{- 1} \int_{\partial \Omega} | \nabla_{\partial \Omega}
   \psi |^2 , \label{eq cs 1}
 \end{equation}
and
\begin{equation} \gamma w_{\nu}^2 - \gamma h w_{\nu} + h^2 = \gamma (w_{\nu} - \tfrac{1}{2}
   h)^2 + (1 - \tfrac{1}{4} \gamma) h^2  \geq (1 - \tfrac{1}{4}
   \gamma) h^2 ,  \end{equation}
which is positive by the assumption $0 < \gamma < 4$. Therefore,
\begin{align}
0 \leq & (1 + \tfrac{1}{4} \gamma (1 - \gamma / 4)^{- 1})
\int_{\partial \Omega} | \nabla_{\partial \Omega} \psi |^2 + \int_{\partial
\Omega} \operatorname{Sc}_{\partial \Omega} \psi^2 \\
& \quad - \int_{\partial \Omega} \left[ (1 - \tfrac{1}{4} \gamma) h^2 - |
\nabla h| +  \Lambda \right] \psi^2 .
\label{stability for h}
\end{align}
Then by the ODE \eqref{eq ode spec ric} and that \( \operatorname{Lip} \chi < 1 \),
\begin{align}
& \tfrac{4}{4-\gamma}
\int_{\partial \Omega} | \nabla_{\partial \Omega} \psi |^2 + \int_{\partial
\Omega} \operatorname{Sc}_{\partial \Omega} \psi^2 \\
>&  \int_{\partial \Omega} \left[ (1 - \tfrac{1}{4} \gamma) (\eta\circ \chi)^2 +\eta'\circ \chi +  \Lambda \right] \psi^2 =0.
\end{align}
Since $\partial \Omega$ is homologous to $T^{2}$, we can choose a connected component of $\partial \Omega$ which admits a map of non-zero 
degree to \( T^2 \). So \( \chi (\Sigma) \leq 0 \).
By taking \( \psi =0 \) on other components and \( \psi =1 \) on $\Sigma$, and applying the Gauss-Bonnet theorem, we obtain a contradiction
to the above inequality.
\end{proof}



\subsection{Existence of a non-trivial minimiser} \label{sub exist ricci}
From here in this section, we are devoted to the proof of Theorem \ref{ricci rigid reformulation}. First, 
we construct a non-trivial minimiser to the weighted functional
\eqref{eqn-omega} by using an argument of J. Zhu
\cite{zhu2021}.

\begin{lemma}\label{lm existence minimiser spectral ricci}
  Let \( M \) be as in Theorem \ref{ricci rigid reformulation}, then there exists a minimiser \( \Omega \)
  of the functional \eqref{eqn-omega}.
  \end{lemma}
\begin{proof}
  For convenience, we multiply \( \zeta \) by \( \ell_0 \) and still denote the new resulting function by \( \zeta \).
Hence,
\[ \zeta (\partial_{\pm} M) = \pm \ell_0, \text{ and }  |\nabla \zeta| \leq 1. \]
We 
choose an odd, smooth function $\alpha (t) : [- \ell_0, \ell_0] \to \mathbb{R}$ such
that $\alpha (t) > 0$ on $(0, \ell_0]$, $\alpha' (t) > 0$ on $[0,
\tfrac{\ell_0}{2})$, $\alpha' (t) < 0$ on $(\tfrac{\ell_0}{2}, \ell_0]$. 

Let $\eta$ be the function given in \eqref{eq ode spec ric}. 
We define $\eta_{\varepsilon} (t) = \eta (t + \varepsilon \alpha (t))$ on a
sub-interval $(- T_{\varepsilon}, T_{\varepsilon})$ of $[- \ell_0, \ell_0]$ such
that $\eta_{\varepsilon} (t) \to \pm \infty$ as $t \to \mp T_{\varepsilon}$,
and we easily find that
\[ (1 - \tfrac{1}{4} \gamma) \eta_{\varepsilon}^2 + \eta_{\varepsilon}' +
   \Lambda = \varepsilon \alpha' (t) \eta' (t + \varepsilon \alpha (t)), \]
and
\begin{align}
  (1 - \tfrac{1}{4} \gamma) \eta_{\varepsilon}^2 + \eta_{\varepsilon}' +
  \Lambda & > 0 \text{ if } \tfrac{\ell_0}{2} < |t| < T_{\varepsilon},
  \label{crucial of perturbation} \\
  (1 - \tfrac{1}{4} \gamma) \eta_{\varepsilon}^2 + \eta_{\varepsilon}' +
  \Lambda & < 0 \text{ if } |t| < \tfrac{\ell_0}{2} . 
\end{align}

Define $h_{\varepsilon} (x) = \eta_{\varepsilon} (\zeta (x))$, where \( x\in M \).
Fix a sufficiently small number \( \varepsilon_0 >0 \). By Sard's lemma, \( \zeta^{-1}( \pm T_{\varepsilon}) \) are both regular surfaces of \( M \) for almost all \( \varepsilon \in (0,\varepsilon_0) \).
We use such \( \varepsilon \). Due to the
condition $\eta_{\varepsilon} (t) \to \pm \infty$ as $t \to \mp
T_{\varepsilon}$, and the existence
result of Lemma \ref{lem-warped-bubble}, we can construct a minimising \( h_{\varepsilon} \)-bubble \( \Omega_{\varepsilon} \subset \zeta^{-1}((-T_{\varepsilon}, T_{\varepsilon})) \).
Note $\partial\Omega_{\varepsilon}$ and $h_{\varepsilon}$ satisfy the inequality
\begin{align}
0 \leq & (1 + \tfrac{1}{4} \gamma (1 - \gamma / 4)^{- 1})
\int_{\partial \Omega_\varepsilon} | \nabla_{\partial \Omega_\varepsilon} \psi |^2 + \int_{\partial
\Omega_{\varepsilon}} \operatorname{Sc}_{\partial \Omega_\varepsilon} \psi^2 \\
& \quad - \int_{\partial \Omega_\varepsilon} \left[ (1 - \tfrac{1}{4} \gamma) h_\varepsilon^2 + 
\nabla_{\nu_\varepsilon} h_\varepsilon +  \Lambda \right] \psi^2 ,
\label{stability for h epssilon}
\end{align}
by the proof of Theorem \ref{Thm-Ricci}.
We claim that
$\partial\Omega_{\varepsilon}$ cannot lie entirely in the region
\[\{x \in M :
\text{ } \tfrac{\ell_0}{2} < | \zeta (x) | < T_{\varepsilon} \}.\] 
That is,
$\partial \Omega_{\varepsilon}$ has a non-empty intersection with the compact set \[K :=
\{x \in M : | \zeta(x) | \leq \tfrac{\ell_0}{2} \}.\]

Indeed, we can show this by selecting a connected component of \( \partial \Omega \).
Since \( \partial \Omega \) is homologous to a level set of \( \zeta \), we can pick a connected component \( \Sigma_{\varepsilon} \) such that the projection of \( \Sigma_{\varepsilon} \) to the \( T^2 \)-factor is a map of non-zero degree.  Taking \( \psi =1 \) on \( \Sigma_{\varepsilon} \) and \( \psi=0 \) in \eqref{stability for h epssilon} on other components, applying \eqref{crucial of perturbation}, and applying the Gauss-Bonnet theorem yields \( \chi (\Sigma_{\varepsilon}) >0 \) which contradicts the choice of \( \Sigma_{\varepsilon} \).
By the curvature estimates (see \cite[Theorem 3.6]{zhou-existence-2020}) and compactness, we can pick a sequence $\varepsilon_k
\to 0$ such that $\Sigma_{\varepsilon_{k}}$ converges locally and smoothly
to a surface $\hat{\Sigma}$ such that $\hat{\Sigma}$ has non-empty
intersection with $K$. 
We see that
each $\Sigma_{\varepsilon_{k}}$ satisfies the inequality
\begin{align}
  \int_{{\Sigma_{\varepsilon_{k}}} } (\langle\nu_{\varepsilon_{k}}, \nabla \zeta\rangle - 1)
  \eta_{\varepsilon_{k}}' \circ \zeta \leq & - \int_{\Sigma_{\varepsilon_{k}}} \left[ (1 -
  \tfrac{\gamma}{4}) (\eta_{\varepsilon_{k}} \circ \zeta)^2 + \Lambda +
  \eta_{\varepsilon_k}' \circ \zeta \right] \\
  = & - \varepsilon_k \int_{\Sigma_{\varepsilon_{k}}} \alpha' (\zeta) \eta' (\zeta + \varepsilon
  \alpha (\zeta)) \\
  = & - \varepsilon_k \left( \int_{\Sigma_{\varepsilon_{k}} \cap K} + \int_{\Sigma_{\varepsilon_{k}}
  \backslash K} \right) \alpha' (\zeta) \eta' (\zeta + \varepsilon \alpha
  (\zeta)) \\
  \leq & C \varepsilon_k \ensuremath{\operatorname{Area}} (\Sigma_{\varepsilon_{k}} \cap
  K) \to 0, 
\end{align}
as $k \to \infty$ where we have dropped the negative part, that is, the
integration on $\Sigma_{\varepsilon_{k}} \backslash K$. Note that $\eta_{\varepsilon_k}' < 0$
for sufficiently large $k$ and $\langle \nu_{\varepsilon_{k}}, \nabla \zeta
\rangle \leq 1$. We can conclude that the limit of $\nu_{\varepsilon_{k}}$
exists and we denote it by $\nu$. Hence $\mathrm{d} \zeta (\nu) = 1$, and so
$\nabla_{\Sigma} \zeta = 0$ by $\operatorname{Lip}\zeta\leq 1$. Therefore the surface $\Sigma$
lies in the level set $\zeta^{- 1} (t_0)$ for some $t_0 \in [- \tfrac{\ell_0}{2},
\tfrac{\ell_0}{2}]$.
\end{proof}

\subsection{Construction of a foliation}
For convenience, we define the notion of infinitesimal rigidity.
First, we show that \( \Sigma \) is infinitesimally rigid (see Lemma \ref{lm infi rig spec ric}), then we apply the inverse function theorem to obtain a foliation near \( \Sigma \). 
\begin{definition} \label{def inf rig spec ric}
  We call a surface \( \Sigma \) infinitesimally rigid if \( \Sigma \) satisfies all the following relations:
  \begin{align}
    \operatorname{Ric} (e_1, e_1) & = \operatorname{Ric} (e_2, e_2)
    = \operatorname{Ric}, \\
    - \gamma u^{- 1} \Delta u + 2\operatorname{Ric} & = \Lambda,
    \\
    \ensuremath{\operatorname{Sc}}_{\Sigma} & = 0, \\
    H & = - \gamma w_{\nu} + h, \\
    (1 - \tfrac{1}{4} \gamma) h^2 + h_{\nu} + \Lambda & = 0, \text{ } h' = - |
    \nabla h|, \text{ } \nabla \zeta = \nu, \\
    w = \log u, w_{\nu} = \tfrac{1}{2} h & , \text{ } h \text{ are constants
    along } \Sigma 
  \end{align}
  are satisfied along \( \Sigma \).
  \end{definition}

\begin{lemma} \label{lm infi rig spec ric}
  Let $\Omega$ be constructed in Lemma \ref{lm existence minimiser spectral
  ricci}, then $\Sigma = \partial \Omega$ is infinitesimally rigid.
\end{lemma}
\begin{proof}
  First, according to \eqref{stability for h}, $\Sigma$ satisfies the
following inequality
\begin{align}
  0 \leq & (1 + \tfrac{1}{4} \gamma (1 - \gamma / 4)^{- 1}) \int_{\Sigma}
  | \nabla_{\Sigma} \psi |^2 + \int_{\Sigma}
  \operatorname{Sc}_{\Sigma} \psi^2 \\
  & \quad - \int_{\Sigma} \left[ (1 - \tfrac{1}{4} \gamma) h^2 + h_{\nu} +
  \Lambda \right] \psi^2 = :-\int_{\Sigma} \psi L \psi =: B (\psi, \psi)
  \label{eq ineq with L B} 
\end{align}
for all $\psi \in C^{\infty} (\Sigma)$. Here,
\[ L = - (1 + \tfrac{1}{4} \gamma (1 - \gamma / 4)^{- 1}) \Delta_{\Sigma}
   +\ensuremath{\operatorname{Sc}}_{\Sigma} - \left[ (1 - \tfrac{1}{4} \gamma)
   h^2 + h_{\nu} + \Lambda \right] . \]
In the following, we show that $\ensuremath{\operatorname{Sc}}_{\Sigma} = 0$, $\nabla \zeta = \nu$
and $(1 - \tfrac{1}{4} \gamma) h^2 + h_{\nu} + \Lambda = 0$ using a now
standard argument due to \cite{fischer-colbrie-structure-1980}. 
By taking $\psi = 1$ on $\Sigma$ and applying the Gauss-Bonnet theorem in
\eqref{eq ineq with L B},
\[ 4 \pi \chi (\Sigma) \geq \int_{\Sigma}
   \ensuremath{\operatorname{Sc}}_{\Sigma} \geq \int_{\Sigma} \left[ (1 -
   \tfrac{1}{4} \gamma) h^2 + h_{\nu} + \Lambda \right] \geq 0. \]
By construction, the Euler characteristic of $\Sigma$ satisfies the bound
$\chi (\Sigma) \leq 0$. The last inequality follows from
$h_{\nu} \geq \eta' | \nabla \zeta| \geq \eta'$ and the ODE \eqref{eq ode spec ric}.
Hence all inequalities in the above
have to be equalities, and we obtain that $\chi (\Sigma) = 0$ and $\nabla \zeta =
\nu$. By considering $\psi = 1$ in $B (\psi, \psi) \geq 0$, we obtain
that $B (1, 1) = 0$. Hence, constants are the first eigenfunction of the
operator $L$ with zero as the eigenvalue, that is, $L 1 = 0$. Hence
$\ensuremath{\operatorname{Sc}}_{\Sigma} = 0$. 

The rest of the infinitesimal rigidity
follows easily by tracing back the derivation of \eqref{eq ineq with L B}.
\end{proof}

Now we use the inverse function theorem to construct a foliation near \( \Sigma \).

\begin{lemma}\label{Lem-foliation-construct}
  Let \( \Omega \) be a stable critical point of $E(\Omega)$ and \( \Sigma= \partial \Omega \). We can construct a local foliation $\{\Sigma_{t}\}_{-\epsilon< t< \epsilon}$ such that $\Sigma_{t}$ is of
  constant \( H+\gamma w_{\nu} -h \), $\Sigma_{0}=\Sigma$,
  \begin{enumerate}
      \item each $\Sigma_{t}$ is a graph over $\Sigma$ with graph function $\rho_{t}$ along outward unit normal vector field $\nu$ such that
      \begin{equation}
          \left.\frac{\partial \rho_{t}}{\partial t}\right|_{t=0}=1 \text{ and } \frac{1}{\operatorname{vol}(\Sigma)}\int_{\Sigma}\rho_{t}dv=t;
      \end{equation}
      \item and $H_{t}+\gamma\omega_{\nu}-h$ is constant on $\Sigma_{t}$.
  \end{enumerate}
  Here, \( d v \) is the volume element of \( \Sigma \) and $H_{t}$ is the mean curvature of $\Sigma_{t}$.
\end{lemma}
\begin{proof}
  Let
  \[\hat{C}^{\alpha}(\Sigma)=\left\{\phi\in C^{\alpha}(\Sigma): \text{ }\int_{\Sigma}\phi dv =0\right\}\]
  for some $\alpha\in(0,1)$.
  Consider the map
  \begin{eqnarray}
    \tilde{\Phi}:C^{2,\alpha}(\Sigma)&\to& \hat{C}^{\alpha}(\Sigma)\times \mathbb{R},\\
     \rho&\mapsto& \left(H_{\rho}+\gamma\omega_{\nu}-h-\frac{1}{\operatorname{vol}(\Sigma)}{\int_{\Sigma}\left(H_{\rho}+\gamma\omega_{\nu}-h\right)dv}, \frac{1}{\operatorname{vol}(\Sigma)}{\int_{\Sigma}\rho dv}\right),
  \end{eqnarray}
  where $H_{\rho}$ is the mean curvature of the graph over $\Sigma$ with graph function $\rho$.
%
%
%
  We use the infinitesimal rigidity of \( \Sigma \) (see Definition \ref{def inf rig spec ric}) to
calculate the first variation of $H + \gamma u^{- 1} u_{\nu} - h$:
\begin{align}\label{subsec-rigidity-ana}
  & \delta_{\phi \nu} (H + \gamma u^{- 1} u_{\nu} - h) \\
  = & - \Delta_{\Sigma} \phi - (\ensuremath{\operatorname{Ric}}(\nu,\nu) + |A|^2)
  \phi - \gamma u^{- 2} u_{\nu}^2 \phi - \gamma u^{- 1} \langle \nabla_{\Sigma} u, \nabla_{\Sigma} \phi \rangle \\
  & \quad + \gamma u^{- 1} (\Delta u - \Delta_{\Sigma} u - H u_{\nu}) \phi - h_{\nu} \phi \\
  = & - \Delta_{\Sigma} \phi -\ensuremath{\operatorname{Ric}} (e_1,e_{1}) \phi
  -\ensuremath{\operatorname{Ric}} (e_2,e_{2}) \phi + \ensuremath{\operatorname{Sc}}_{\Sigma} \phi - H^2 \phi - \gamma w_{\nu}^2 \phi \\
  & \quad + \gamma u^{- 1} \phi \Delta u - \gamma u^{- 1} \phi
  \Delta_{\Sigma} u - \gamma w_{\nu} (h - \gamma w_{\nu}) \phi - h_{\nu} \phi \\
  = & - \Delta_{\Sigma} \phi - \Lambda \phi  - (h - \gamma \tfrac{1}{2}
  h)^2 \phi - \gamma (\tfrac{1}{2} h)^2 \phi \\
  & \quad  - \gamma (\tfrac{1}{2} h) (h - \tfrac{1}{2} \gamma h) \phi - h_{\nu} \phi \\
  = & - \Delta_{\Sigma} \phi - ((1 - \tfrac{1}{4} \gamma) h^2 + h' + \Lambda)
  \phi \\
  = & - \Delta_{\Sigma} \phi . 
\end{align}
Then, the linearization of $\tilde{\Phi}$
at $\rho=0$
given by
  \begin{eqnarray}
    D\tilde{\Phi}|_{\rho=0}: C^{2,\alpha}(\Sigma)\to \hat{C}^{\alpha}(\Sigma)\times \mathbb{R},\quad \psi\mapsto \left(\Delta_{\Sigma} \psi,\frac{1}{\operatorname{vol}(\Sigma)}\int_{\Sigma}\psi dv\right),
  \end{eqnarray}
  is invertible. By the inverse function theorem, we can find a family of functions $\rho_{t}:\Sigma\to \mathbb{R}$ with $t\in (-\epsilon,\epsilon)$ with the following properties:
  \begin{enumerate}
    \item the function $\rho_{t}$ satisfies $\rho_{0}\equiv0$,
    \begin{equation}\label{Eqn-construct-rho}
      \left.\frac{\partial \rho_{t}}{\partial t}\right|_{t=0}\equiv 1, \text{ and } \frac{1}{\operatorname{vol}(\Sigma)}\int_{\Sigma}\rho_{t} dv =t
    \end{equation}
  \item the graphs $\Sigma_{t}$ over $\Sigma$ with the graph function $\rho_{t}$ is of constant \( H + \gamma u^{-1} u_{\nu} -h \).
  \end{enumerate}
  From \eqref{Eqn-construct-rho}, with the value of $\epsilon$ decreased a little bit, the speed $\partial_t\rho_{t}$ will be positive for \( t \in (-\epsilon, \epsilon) \),
from which it follows that the graphs $\{\Sigma_{t}\}_{t\in (-\epsilon,\epsilon)}$ form a foliation around $\Sigma$.
\end{proof}

\begin{lemma}
  Let $\tilde{\Omega}_{t}$ be the region enclosed by $\Sigma_{t}$ and $\Sigma$, and 
  \begin{eqnarray}
\Omega_{t}=\begin{cases}
    \Omega \cup \tilde{\Omega}_{t},\mbox{ if } 0<t< \epsilon,\\
    \Omega \setminus \tilde{\Omega}_{t},\mbox{ if } -\epsilon <t<0.
\end{cases}    
\end{eqnarray} 
Then $\Omega_{t}$ is also a minimiser of $E(\Omega)$.
\end{lemma}
\begin{proof}
    The first step is to prove that $H_{t}+\gamma\omega_{\nu}-h=0$ on $\Sigma_{t}$ for $t\in (-\epsilon,\epsilon)$. Then the stable minimal property of $\Omega_{t}$ follows. We do that in Proposition \ref{Prop-monotonic-formula} and \ref{Prop-minimiser}.
\end{proof}

\begin{proposition}\label{Prop-monotonic-formula}
There exists a continuous function $\Psi (t)$ such
  that
  \begin{equation}
    \tfrac{\mathrm{d}}{\mathrm{d} t} \left( \exp (\int_0^t \Psi (\tau)
    \mathrm{d} \tau) \tilde{H} \right) \leq 0 \label{eq:H tilde inequality}
  \end{equation}
  where
  \begin{equation}
    \tilde{H} = H + \gamma w_{\nu} - h. \label{H tilde}
  \end{equation}
\end{proposition}

\begin{proof}
    Let $\Phi : \Sigma \times (- \varepsilon,
\varepsilon) \to M$ parametrize the local foliation, $Y = \tfrac{\partial
\Phi}{\partial t}$, and $\phi_t = \langle Y, \nu_t \rangle$. Since we have shown
that $\phi_0$ is a constant. We can fix $\varepsilon$ sufficiently small so that
$\phi_t > 0$ for all $t \in (- \varepsilon, \varepsilon)$. Recall that the
first variation gives
\begin{align}
  & - \tilde{H}' (t) \\
  = & - \tfrac{\mathrm{d}}{\mathrm{d} t} (H_t + \gamma w_{\nu} - h)
  \\
  = & \Delta_{\Sigma_t} \phi_t + (\ensuremath{\operatorname{Ric}}(\nu_t,\nu_t) +
  |A_t |^2) \phi_t + \gamma u^{- 2} u_{\nu}^2 \phi_t \\
  & \quad - \gamma u^{- 1} (\Delta u - \Delta_{\Sigma_t} u - H u_{\nu})
  \phi_t + \gamma u^{- 1} \langle \nabla_{\Sigma_t} u, \nabla_{\Sigma_t}
  \phi_t \rangle \\
  & \quad + h_{\nu_t} \phi_t . 
\end{align}
Using the rewrite \eqref{eq:zhu 5.2}, the definition of $\operatorname{Ric}$,
\eqref{H tilde}, and the spectral bound \eqref{eq spectral ricci}, and with
suitable grouping of terms, we see
\begin{align}
  & - \tilde{H}' (t) \phi_t^{- 1} \\
  \geq & (\phi_t^{- 1} \Delta_{\Sigma_t} \phi_t + \gamma  u^{- 1}
  \Delta_{\Sigma_t} u + \gamma u^{- 1} \langle \nabla_{\Sigma_t} u,
  \nabla_{\Sigma_t} \phi_t \rangle \phi_t^{- 1}) \\
  & \quad + \Lambda - \ensuremath{\operatorname{Sc}}_{\Sigma_t} + h_{\nu_t} + H^2 + \gamma w_{\nu}^2 + \gamma H w_{\nu} . 
\end{align}
Inserting \eqref{H tilde} in the above and using
\begin{align}
  & (\phi_t^{- 1} \Delta_{\Sigma_t} \phi_t + \gamma  u^{- 1}
  \Delta_{\Sigma_t} u + \gamma u^{- 1} \langle \nabla_{\Sigma_t} u,
  \nabla_{\Sigma_t} \phi_t \rangle \phi_t^{- 1}) \\
  = & \ensuremath{\operatorname{div}}_{\Sigma_t} \left(
  \frac{\nabla_{\Sigma_t} \phi_t}{\phi_t} + \gamma \frac{\nabla_{\Sigma_t}
  u}{u} \right) + (1 - \tfrac{\gamma}{4}) \left| \tfrac{\nabla_{\Sigma_t}
  \phi_t}{\phi_t} \right|^2 + \gamma \left| \frac{\nabla_{\Sigma_t} u}{u} +
  \frac{\nabla_{\Sigma_t} \phi_t}{2 \phi_t} \right|^2 \\
  \geq & \ensuremath{\operatorname{div}}_{\Sigma_t} \left(
  \frac{\nabla_{\Sigma_t} \phi_t}{\phi_t} + \gamma \frac{\nabla_{\Sigma_t}
  u}{u} \right) 
\end{align}
yields
\begin{align}
  & - \tilde{H}' (t) \phi_t^{- 1} \\
  \geq & \ensuremath{\operatorname{div}}_{\Sigma_t} \left(
  \frac{\nabla_{\Sigma_t} \phi_t}{\phi_t} + \gamma \frac{\nabla_{\Sigma_t}
  u}{u} \right) \\
  & \quad + \Lambda - \ensuremath{\operatorname{Sc}}_{\Sigma_t} + h_{\nu_t} + \tilde{H}^2 + \tilde{H} (-
  \gamma w_{\nu} + 2 h) + (\gamma w_{\nu}^2 + h^2 - \gamma h w_{\nu}) .
\end{align}
Applying the elementary inequality
\[ \gamma w_{\nu}^2 + h^2 - \gamma h w_{\nu} = \gamma (w_{\nu} - \tfrac{1}{2}
   h)^2 + (1 - \tfrac{1}{4} \gamma) h^2 \geq (1 - \tfrac{1}{4} \gamma)
   h^2 \]
on the last term on the bracket and the trivial bound $\tilde{H}^2 \geq
0$, we obtain that
\begin{align}
  & - \tilde{H}' (t) \phi_t^{- 1} \\
  \geq & \tilde{H} (- \gamma w_{\nu} + 2 h)
  +\ensuremath{\operatorname{div}}_{\Sigma_t} \left( \frac{\nabla_{\Sigma_t}
  \phi_t}{\phi_t} + \gamma \frac{\nabla_{\Sigma_t} u}{u} \right) -
  \ensuremath{\operatorname{Sc}}_{\Sigma_t} \\
  & \quad + \left( \Lambda + h_{\nu_t} + (1 - \tfrac{1}{4} \gamma) h^2
  \right) . 
\end{align}
Note that $\Lambda + h_{\nu_t} + (1 - \tfrac{1}{4} \gamma) h^2 \geq 0$,
and finally,
\[ - \tilde{H}' (t) \phi_t^{- 1} \geq \tilde{H} (- \gamma w_{\nu} + 2 h)
   +\ensuremath{\operatorname{div}}_{\Sigma_t} \left( \frac{\nabla_{\Sigma_t}
   \phi_t}{\phi_t} + \gamma \frac{\nabla_{\Sigma_t} u}{u} \right) -
   \ensuremath{\operatorname{Sc}}_{\Sigma_t} . \]
We integrate the above on $\Sigma_t$ and we find by the divergence theorem and
Gauss-Bonnet theorem ($\Sigma_t$ is a torus since it is a small perturbation of \( \Sigma \)) that,
\[ - \tilde{H}' (t) \int_{\Sigma_t} \tfrac{1}{\phi_t} \geq \tilde{H} (t)
   \int_{\Sigma_t} (- \gamma w_{\nu} + 2 h) . \]
We set $\Psi (t) = (\int_{\Sigma_t} \tfrac{1}{\phi_t})^{- 1} \int_{\Sigma_t}
(- \gamma w_{\nu_t} + 2 h)$, then
\[ \tilde{H}' + \Psi (t) \tilde{H} \leq 0. \]
By solving this inequality, we finish the proof of the
lemma.
\end{proof}

\begin{proposition}\label{Prop-minimiser}
  Every $\Omega_t$ is a minimiser to $E (\Omega)$.
\end{proposition}
\begin{proof}
  Let \( \partial_t \) be the variation vector field of the foliation \( \{\Sigma_t \}_{t\in (-\varepsilon,\varepsilon)} \), and \( \phi_t = \langle \partial_t , \nu_t \rangle \).
    Recall that the first variation of $E$ is given in Lemma \ref{lem-first-variation}, and
\[ \tfrac{\mathrm{d}}{\mathrm{d} t} E (\Omega_t) = \int_{\Sigma_t} \tilde{H}
   (t) u^{\gamma} \phi_{t}. \]
 It follows from $\tilde{H} (0) = 0$ and
\eqref{eq:H tilde inequality} that $\tilde{H} (t) \leq 0$ for $t
\geq 0$ and $\tilde{H} (t) \geq 0$ for $t \leq 0$. So $E
(\Omega_t) \leq E (\Omega_0)$ for all $t \in (- \varepsilon,
\varepsilon)$ and hence
\[ E (\Omega_t) = E (\Omega_0)  \]
for all \( t \in (-\varepsilon, \varepsilon) \).
Hence, all the foliation analysis on $\Omega_0$ can be applied to $\Omega_t$.
Since $M$ is connected, we can conclude that $M$ is foliated by the boundaries
$\partial \Omega_t$ of the minimisers of $E
$.
\end{proof}


\subsection{Rigidity}

The foliation indicates that the band is locally isometric to $(-\epsilon,\epsilon)\times \Sigma $ with the metric $g=\phi_{t}(x)dt^2+g_{t}$, where $g_{t}$ is a flat metric on $\Sigma$, $x\in \Sigma$, and $\phi_{t}$ is the function constructed in the proof of Proposition \ref{Prop-monotonic-formula}. Since $\Omega_{t}$ is stable, the computation in the proof of Lemma \ref{Lem-foliation-construct} shows that
\[\Delta_{g_{t}}\phi_{t}(x)=0.\]
It then follows that $\phi_{t}$ is constant on each $\Sigma_{t}$.  Since 
\begin{align*}
  2\operatorname{Ric}=&\Lambda+{ \gamma u^{- 1}}  \Delta u,\\
  =&\Lambda+{ \gamma u^{- 1}}  \left(\Delta_{\Sigma}+H \partial_{t}+\partial^2_{t}-\nabla_{\partial_{t}}\partial_{t}\right) u,
\end{align*}
we see that $H$ and $u$ depend only on $t$. Therefore, $\operatorname{Ric}$ depends only on $t$.

Recall that the local orthonormal basis $\{\nu=\partial_{t},e_{1},e_{2}\}$ on $M$ is chosen such that 
\[\operatorname{Ric}(\nu,e_{\alpha})=0, \quad \operatorname{Ric}(e_{1},e_{2})=0\]
for $\alpha=1,2$. By Schoen-Yau's rewrite \eqref{eq sy rewrite pre} of the Gauss equation, using $\operatorname{Ric}(e_{1},e_{1})=\operatorname{Ric}(e_{2},e_{2})=\operatorname{Ric}$, and by the definition of scalar curvature
\( \operatorname{Sc}_g = \operatorname{Ric}(\nu,\nu) + 2\operatorname{Ric} \),
we have
\begin{align*}
0=&\operatorname{Sc_{g}}-2\operatorname{Ric}(\nu,\nu)+H^2-|A|^2    \\
=&-\operatorname{Sc_{g}}+4 \operatorname{Ric}+H^2-|A|^2.
\end{align*}
By the contracted Bianchi identity, 
\begin{align*}
    \frac{1}{2}e_{1}(\operatorname{Ric}(\nu,\nu))=&\frac{1}{2} e_{1} (\operatorname{Sc})\\
    =&\nabla_{\nu}\operatorname{Ric}(\nu,e_{1})+\sum_{\alpha=1}^{2}\nabla_{e_{\alpha}}\operatorname{Ric}(e_{1},e_{\alpha})\\
    =&-\operatorname{Ric}(\nabla_{\nu}\nu,e_{1})-\operatorname{Ric}(\nabla_{\nu}e_{1},\nu)-\sum_{\alpha=1}^{2}\operatorname{Ric}(\nabla_{e_{\alpha}}e_{1},e_{\alpha})\\
    &\quad-\sum_{\alpha=1}^{2}\operatorname{Ric}(e_{1},\nabla_{e_{\alpha}}e_{\alpha})\\
    =&0.
\end{align*}
Similarly, $e_{2}(\operatorname{Ric}(\nu,\nu))=0$. Thus, $\operatorname{Ric}(\nu,\nu)$ only depends on $t$, and so does $|A|^2$. By 
\begin{align*}
    R(\nu,e_{1},e_{1},\nu)+R(e_{2},e_{1},e_{1},e_{2})=&\operatorname{ Ric}(e_{1},e_{1})=\operatorname{Ric}\\
    R(\nu,e_{2},e_{2},\nu)+R(e_{2},e_{1},e_{1},e_{2})=&\operatorname{ Ric}(e_{2},e_{2})=\operatorname{Ric},
\end{align*}
we see $R(\nu,e_{1},e_{1},\nu)=R(\nu,e_{2},e_{2},\nu)$.

However $\operatorname{Ric}(\nu,\nu)=R(\nu,e_{1},e_{1},\nu)+R(\nu,e_{2},e_{2},\nu)$, then $R(\nu,e_{1},e_{1},\nu)$ and $R(\nu,e_{2},e_{2},\nu)$ also only depend on $t$.
Again because the frame \( \{e_1,e_2\} \) is arbitrary, we can conclude that \( R(e_1,\nu,\nu,e_2 )=0 \).
We also have the following equation 
\begin{align*}
    \begin{cases}
        \partial_{t}A-A^2=-R(\cdot,\nu,\nu,\cdot);\\
        \partial_{t}g_{t}=2A.
    \end{cases}
\end{align*}
Because of the above ODE, and that \( R(e_1,\nu,\nu,e_1) = R(e_2,\nu,\nu,e_2) \) depend only on \( t \) and that
\( R(e_1,\nu,\nu,e_2) =0 \), we can write the metric 
$g=dt^2+\phi^2_{1} ds_{1}^2+\phi_{2}^2ds_{2}^2$ for some positive functions \( \phi_1 \) and \( \phi_2 \)
which only depend 
on \( t \).
The calculation of explicit forms of \( \phi_1 \) and \( \phi_2 \) are given in Appendix \ref{subsec-determine-the-metric}.

\section{Band width estimate with the spectral scalar curvature condition}\label{sec spectral scalar}

In this section, we prove Theorems \ref{Thm-scalar} and \ref{thm rigid scalar}.
The proof of Theorem \ref{thm rigid scalar} is similar to that of Theorem \ref{ricci rigid reformulation},
in particular for three dimensions. However, we need to modify some of the arguments to adapt to higher dimensions.

\subsection{Proof of the band width estimate}
Before we proceed, we define an auxiliary function 
\begin{equation} \label{eq eta in spec scalar}
  \eta(t):=-
\frac{\sqrt{\Lambda}}{\sqrt{\frac{- n \gamma  + \gamma  + 2
n}{4(n-1)+2\gamma(2-n)}}}\tan\left( t \sqrt{\tfrac{- n \gamma  + \gamma  + 2
n}{4(n-1)+2\gamma(2-n)}\Lambda}\right),
\end{equation}
we can check that \( \eta \) satisfies the ODE
\begin{equation} \label{eq eta ode spec scalar}
  \frac{- n \gamma  + \gamma  + 2
n}{4(n-1)+2\gamma(2-n)} \eta^2 + \eta'
+  \Lambda =0 \text{ and } \eta'<0. \end{equation}

\begin{proof}[Proof of Theorem \ref{Thm-scalar}]
  We prove by contradiction. We assume that
  \[\operatorname{width}_{g}(M)> 2\ell_{1}:=\frac{\pi}{\sqrt{\frac{- n \gamma  + \gamma  + 2
          n}{4(n-1)+2\gamma(2-n)}\Lambda}},\]
  then there exists a small $\delta>0$ such that
  the band
  \[ M\setminus\{x\in M: \text{ }\operatorname{dist}_{g}(x,\partial M)<\delta\}\]
  is of length greater than \( 2 \ell_1 \).
  Then we can further perturb this band into a smooth band \( \tilde M \)
  so that the band width of \(  \tilde M \) is of length greater than \( 2\ell_{1} \). 
  By \cite[Lemma 7.2]{cecchini-scalar-2024}, there exists a smooth function $\chi:\tilde{M}\to [-\ell_{1} ,  \ell_{1}] $
on $\tilde{M}$ such that the Lipschitz constant $\operatorname{Lip}\chi<1$, $\chi^{-1}(- \ell_{1})=\partial_{-}\tilde{M}$ and $\chi^{-1}(  \ell_{1})=\partial_{+}\tilde{M}$. Then the function  
\[h:=\eta \circ \chi \]
on $\tilde{M}\setminus \partial \tilde{M}$ tends to $-\infty$ as $x\to \partial_{+}\tilde{M}$ and to $\infty$ as $x\to \partial_{-}\tilde{M}$, where \( \eta \) is the function given in \eqref{eq eta in spec scalar} and satisfies \eqref{eq eta ode spec scalar}.
We also have that $u$ is bounded below by a positive number on $\tilde{M}$. 
  Using Lemma \ref{lem-warped-bubble}, we have a minimising warped $\mu$-bubble $\Omega$ such that $H=-\gamma u^{- 1} u_{\nu} + h$ along \( \partial \Omega \) and the
  second variation
$\tfrac{\mathrm{d}^2}{\mathrm{d} t^2} E (\Omega_t) |_{t = 0}$ is non-negative for any smooth variations of $\Omega$ such that $\Omega_0 = \Omega$.

Given any function $\psi \in C^{\infty}(\partial \Omega)$, we consider the variations $ \Omega_t$ of $\Omega$ such that normal speed of the variations $\partial \Omega_{t}$ is $\phi=u^{\gamma/2}\psi$ and
we can firstly rewrite the second variation as in
Lemma \ref{lm second var rewrite} and we now rewrite the second variation further based on Lemma \ref{lm second var rewrite}.

Using the following Schoen-Yau's trick of the Gauss equation, we have 
\begin{align}
  |A|^2 +\ensuremath{\operatorname{Ric}} (\nu,\nu) &= \tfrac{1}{2} (\operatorname{Sc}_g - \operatorname{Sc}_{\partial \Omega} + |A|^2 + H^2) \label{eq sy rewrite pre} \\
                                                 & \geq \tfrac{1}{2} (\operatorname{Sc}_g - \operatorname{Sc}_{\partial \Omega} + \frac{n}{n-1} H^2) .\label{eq:sy rewrite}
\end{align}
Using the above, \eqref{eq spectral scalar} and that $H_{\partial \Omega} = - \gamma w_{\nu} + h$ (recall that $w=\log u$) in
Lemma \ref{lm second var rewrite}, we arrive
\begin{align}
0 \leq & \int_{\partial \Omega} | \nabla_{\partial \Omega} \psi |^2 +
\int_{\partial \Omega} [\gamma \psi \langle \nabla_{\partial \Omega} w,
\nabla_{\partial \Omega} \psi \rangle + (\tfrac{\gamma^2}{4} - \gamma)
\psi^2 | \nabla_{\partial \Omega} w|^2] \\ \label{eq similar back ref}
& \quad + \int_{\partial \Omega} [\tfrac{1}{2} \operatorname{Sc}_{\partial \Omega} -
\tfrac{n}{2 (n - 1)} (- \gamma w_{\nu} + h)^2 -  \Lambda] \psi^2
\\ 
& \quad - \int_{\partial \Omega} [\gamma (- \gamma w_{\nu} + h) w_{\nu} +
h_{\nu} + \gamma w_{\nu}^2] \psi^2 \\ 
= & \int_{\partial \Omega} | \nabla_{\partial \Omega} \psi |^2 +
\tfrac{1}{2} \operatorname{Sc}_{\partial \Omega} \psi^2 + \int_{\partial \Omega} [\gamma
\psi \langle \nabla_{\partial \Omega} w, \nabla_{\partial \Omega} \psi
\rangle + (\tfrac{\gamma^2}{4} - \gamma) \psi^2 | \nabla_{\partial \Omega}
w|^2] \\
& \quad - \int_{\partial \Omega} \left[ (\tfrac{n}{2 (n - 1)} \gamma^2 -
\gamma^2 + \gamma) w_{\nu}^2 - \tfrac{1}{n - 1} \gamma h w_{\nu} +
\tfrac{n}{2 (n - 1)} h^2 + h_{\nu}  + \Lambda \right] \psi^2 .
\end{align}
Since $\gamma^2 / 4 - \gamma < 0$, so by Cauchy-Schwarz inequality,
\[ \int_{\partial \Omega} [\gamma \psi \langle \nabla_{\partial \Omega} w,
   \nabla_{\partial \Omega} \psi \rangle + (\tfrac{\gamma^2}{4} - \gamma)
   \psi^2 | \nabla_{\partial \Omega} w|^2] \leq \tfrac{1}{4} \gamma (1 -
   \tfrac{\gamma}{4})^{- 1} \int_{\partial \Omega} | \nabla_{\partial \Omega}
   \psi |^2  \]
 as done in \eqref{eq cs 1}
and 
\begin{align}
& (\tfrac{n}{2 (n - 1)} \gamma^2 - \gamma^2 + \gamma) w_{\nu}^2 -
\tfrac{1}{n - 1} \gamma h w_{\nu} + \tfrac{n}{2 (n - 1)} h^2 \\
\geq & \left[ - \frac{\gamma^2}{4 (\tfrac{n}{2 (n - 1)} \gamma^2 -
\gamma^2 + \gamma) (n - 1)^2} + \tfrac{n}{2 (n - 1)} \right] h^2 \\
= & \frac{- n \gamma  + \gamma  + 2 n}{4 (\tfrac{n}{2 (n - 1)} \gamma  -
\gamma  + 1) (n - 1)} h^2 >0.
\end{align}
The positive sign in the last line is due to the assumption $0 < \gamma < \frac{2n}{n-1}$. Therefore,
\begin{align}
0 \leq & \tfrac{4}{4-\gamma}
\int_{\partial \Omega} | \nabla_{\partial \Omega} \psi |^2 + \tfrac{1}{2}
\int_{\partial \Omega} \operatorname{Sc}_{\partial \Omega} \psi^2 \\
& \quad - \int_{\partial \Omega} \left[ \frac{- n \gamma  + \gamma  + 2
n}{4(n-1)+2\gamma(2-n)} h^2 + h_{\nu}
+  \Lambda \right] \psi^2 \label{eq bilinear energy nonnegative}
\end{align}
by the estimate \( h_{\nu} > -|\nabla h| = \eta'\circ \chi \).
Then
\begin{align}
 \tfrac{4}{4-\gamma}
\int_{\partial \Omega} | \nabla_{\partial \Omega} \psi |^2 + \tfrac{1}{2}
\int_{\partial \Omega} \operatorname{Sc}_{\partial \Omega} \psi^2 
>& 
\int_{\partial \Omega} \left[ \frac{- n \gamma  + \gamma  + 2
n}{4(n-1)+2\gamma(2-n)} h^2 +h'
+  \Lambda \right] \psi^2 \\
=&0.
\end{align}
Because $\frac{4(n-2)}{n-3}>\frac{8}{4-\gamma}$ for $3<n\leq 7$, then the operator $-\frac{4(n-2)}{n-3}\Delta_{\partial\Omega}+\operatorname{Sc}_{\partial \Omega}$ is positive. Let \( u>0 \) be its first eigenfunction.
Since $\partial \Omega$ is homologous to a level set of \( \zeta \), we can choose a connected component of \( \partial \Omega \) such that the projection of \( \Sigma \) to
the \( T^{n-1} \)-factor is a map of non-zero degree.
By the conformal change, \( u^{\tfrac{4}{n-3}}g_{|_{\partial \Omega}} \) is a metric of
positive scalar curvature metric on $\partial \Omega$. 
This is a contradiction to the non-existence of a metric of positive scalar curvature on \( T^{n-1} \).
If $n=3$, we can invoke similar arguments as in Theorem \ref{Thm-Ricci}, which is also a contradiction.
\end{proof}

\subsection{Existence of a non-trivial minimiser}

In this section, we prove our scalar curvature rigidity result (Theorem \ref{thm rigid scalar}).
We construct a non-trivial minimiser to the functional
\eqref{eqn-omega} by modifying an argument of J. Zhu
\cite{zhu-rigidity-2020}.


\begin{lemma} \label{scalar stable existence}
  Let \( M \) be as in Theorem \ref{thm rigid scalar}, then there exists a stable critical point \( \Omega \)
  of the weighted functional \eqref{eqn-omega} which admits a map to \( T^{n-1} \) of non-zero degree.
  \end{lemma}

\begin{proof}
For convenience, we multiply \( \zeta \) by \( \ell_1 \) and still denote the resulting function by \( \zeta \), then
\[ \zeta (\partial_{\pm} M) = \pm   \frac{\pi}{2\sqrt{\frac{- n \gamma  + \gamma  + 2 n}{4(n-1)+2\gamma(2-n)}\Lambda}}=\pm\ell_1, \text{ }
   |\nabla \zeta| \leq 1. \]
We can
choose an odd, smooth function $\alpha (t) : [- \ell_1, \ell_1] \to \mathbb{R}$ such
that $\alpha (t) > 0$ on $(0, \ell_{1}]$, $\alpha' (t) > 0$ on $[0,
\tfrac{\ell_{1}}{2})$, $\alpha' (t) < 0$ on $(\tfrac{\ell_{1}}{2}, \ell_{1}]$.

We define a perturbed version $\eta_{\varepsilon} (t) = \eta (t + \varepsilon \alpha (t))$ of the function \( \eta \) in \eqref{eq eta in spec scalar}
on a
sub-interval $(- T_{\varepsilon}, T_{\varepsilon})$ of $[- \ell_1, \ell_1]$ such
that $\eta_{\varepsilon} (t) \to \pm \infty$ as $t \to \mp T_{\varepsilon}$,
and we easily find that
\begin{equation}
  \frac{- n \gamma  + \gamma  + 2
n}{4(n-1)+2\gamma(2-n)}   \eta_{\varepsilon}^2 + \eta_{\varepsilon}' +
\Lambda = \varepsilon \alpha' (t) \eta' (t + \varepsilon \alpha (t)), \label{eq almost ode}
\end{equation}
and
\begin{align}
 \frac{- n \gamma  + \gamma  + 2
n}{4(n-1)+2\gamma(2-n)}  \eta_{\varepsilon}^2 + \eta_{\varepsilon}' +
  \Lambda & > 0 \text{ if } \tfrac{\ell_1}{2} < |t| < T_{\varepsilon},
  \label{crucial of perturbation 2} \\
 \frac{- n \gamma  + \gamma  + 2
n}{4(n-1)+2\gamma(2-n)}  \eta_{\varepsilon}^2 + \eta_{\varepsilon}' +
  \Lambda & < 0 \text{ if } |t| < \tfrac{\ell_1}{2} . 
\end{align}

Define $h_{\varepsilon} (x) = \eta_{\varepsilon} (\zeta (x))$, where \( x\in M \), and
fix a sufficiently small number \( \varepsilon_0 >0 \). By Sard's lemma, \( \zeta^{-1}( \pm T_{\varepsilon}) \) are both regular surfaces of \( M \) for almost all \( \varepsilon \in (0,\varepsilon_0) \).
We use such \( \varepsilon \). Due to the
condition $\eta_{\varepsilon} (t) \to \pm \infty$ as $t \to \mp
T_{\varepsilon}$, and the existence
result of Lemma \ref{lem-warped-bubble}, we can construct a stable warped \( \mu \)-bubble \( \Omega_{\varepsilon} \subset \zeta^{-1}((-T_{\varepsilon}, T_{\varepsilon})) \).
%
%
Similar to the proof of Theorem \ref{Thm-scalar}, we see
\begin{align}
0 \leq & \tfrac{4}{4-\gamma}
\int_{\partial \Omega_{\varepsilon}} | \nabla_{\partial \Omega_{\varepsilon}} \psi |^2 + \tfrac{1}{2} \int_{\partial
\Omega_{\varepsilon}} \operatorname{Sc}_{\partial \Omega_\varepsilon} \psi^2 \\
& \quad - \int_{\partial \Omega_\varepsilon} \left[  \frac{- n \gamma  + \gamma  + 2
n}{4(n-1)+2\gamma(2-n)}  h_\varepsilon^2 + 
\nabla_{\nu_\varepsilon} h_\varepsilon +  \Lambda \right] \psi^2 .
\label{stability for h epssilon 2}
\end{align}
For \( n=3 \), we can argue similarly as Lemma \ref{lm existence minimiser spectral ricci} and finish the proof. Hence, hereafter, we focus only on the dimensions \( 3 < n \leq 7 \).

We claim that
$\partial \Omega_{\varepsilon}$ cannot lie entirely in the region $\{x \in M :
\text{ } \tfrac{\ell_1}{2} < | \zeta (x) | < T_{\varepsilon} \}$. That is,
$\partial \Omega_{\varepsilon}$ has a non-empty intersection with the compact set
\[K :=
  \{x \in M : | \zeta (x) | \leq \tfrac{\ell_1}{2} \}.
\]
Let $\alpha$ be the positive constant given by
\begin{equation}
  \frac{4 (n - 2)}{n - 3} \alpha =  \frac{8}{4 - \gamma} . \label{eq def
  alpha}
\end{equation}
By the range $0 < \gamma < \tfrac{2 n}{n - 1}$, we have $\alpha < 1$. Let
\[ Q_{\varepsilon} = \frac{- n \gamma + \gamma + 2 n}{4 (n - 1) + 2 \gamma (2
   - n)} h_{\varepsilon}^2 + h_{\nu_{\varepsilon}} + \Lambda, \]
and
\[ L_{\varepsilon} = - \tfrac{4}{4 - \gamma} \Delta_{\partial
   \Omega_{\varepsilon}} + \tfrac{1}{2} \mathrm{Sc}_{\partial
   \Omega_{\varepsilon}} - Q_{\varepsilon} . \]
By \eqref{eq bilinear energy nonnegative}, we see that the first eigenvalue
$\lambda_{1, \varepsilon}$ of the operator $L_{\varepsilon}$ is non-negative.
Assume that $v = v_{\varepsilon}$ is the first eigenfunction, that is,
\begin{equation}
  L_{\varepsilon} v = \lambda_{1, \varepsilon} v. \label{eq eigen}
\end{equation}
We consider the metric $\hat{g}_{\varepsilon} = (v^{\alpha})^{\tfrac{4}{n -
3}} g|_{\partial \Omega_{\varepsilon}}$, by the well known conformal change of
the scalar curvature, it follows that the scalar curvature of the metric
$\hat{g}_{\epsilon}$ is given by
\begin{align}
  (v^{\alpha})^{\frac{n + 1}{n - 3}} \ensuremath{\operatorname{Sc}}_{\partial
  \Omega_{\varepsilon}} (\hat{g}_{\varepsilon}) & = \mathrm{Sc}_{\partial \Omega_{\varepsilon}} v^{\alpha}
  - \tfrac{4 (n - 2)}{n - 3} \Delta_{\partial \Omega_{\varepsilon}} v^{\alpha} \\
  & = v^{\alpha - 1} \left( \mathrm{Sc}_{\partial \Omega_{\varepsilon}} v - \tfrac{4 (n -
  2)}{n - 3} \alpha \Delta_{\partial \Omega_{\varepsilon}} v - \tfrac{4 (n - 2)}{n - 3} \alpha (\alpha - 1)
  v^{- 1} | \nabla^{\partial \Omega_{\varepsilon}} v|^2 \right) . \label{eq full conformal change}
\end{align}
By the definition of $\alpha$ in \eqref{eq def alpha} and that $\alpha < 1$,
we see
\begin{align}
  & v^{\tfrac{4 \alpha}{n - 3} + 1} \ensuremath{\operatorname{Sc}}_{\partial
  \Omega_{\varepsilon}} (\hat{g}_{\varepsilon}) \\
  \geq & \mathrm{Sc}_{\partial \Omega_{\varepsilon}} v - \tfrac{4 (n - 2)}{n - 3}
  \alpha \Delta_{\partial\Omega_{\varepsilon}} v \\
  = & \ensuremath{\operatorname{Sc}}_{\partial \Omega_{\varepsilon}} v - \tfrac{8}{4 -
  \gamma} \Delta_{\partial \Omega_{\varepsilon}} v \\
  = & 2 (\lambda_{1, \varepsilon} v + Q_{\varepsilon} v) \\
  \geq & 2 Q_{\varepsilon} v.  \label{eq lower bound on scalar curvature after conformal change}
\end{align}
If $\partial\Omega_{\varepsilon}$ lies entirely in $\{x \in M : \tfrac{\ell_1}{2} < \zeta
(x) < T_{\varepsilon} \}$, then 
$Q_{\varepsilon} > 0$ by \eqref{crucial of perturbation 2}
and hence $\ensuremath{\operatorname{Sc}}_{\partial
  \Omega_{\varepsilon}} (\hat{g}_{\varepsilon}) > 0$.
By construction of \( \partial \Omega_{ \varepsilon }\), there is a connected component \( \Sigma_{\varepsilon} \)
of \(\partial \Omega_{\varepsilon}  \) such that the projection of \( \Sigma_{\varepsilon} \) to the \( T^{n-1} \)-factor of \( M \) is a map of non-zero degree, and so \( \Sigma_{\varepsilon} \) cannot have a metric of positive scalar curvature which is a contradiction.
Hence, the claim is proved.

Now we consider \( (\Sigma_{\varepsilon},  \hat{g}_{\varepsilon} ) \).
We need to ensure that \( \hat{g}_{\varepsilon} \) has a limit as $\varepsilon \to 0$, and by rescaling 
we can always assume that \( \sup_{K\cap \Sigma_{\varepsilon}} v_{\varepsilon} =1 \).
Then there exists some point \( p_{\varepsilon} \in K\cap \Sigma_{\varepsilon} \) such that \( v_{\varepsilon}(p_{\varepsilon})=1 \) by compactness of \( K \).

By the curvature estimates of \cite[Theorem 3.6]{zhou-existence-2020}, there exists a sequence \( \{ \varepsilon_{k} \}_{k\in \mathbb{N}} \) such that \( \varepsilon_{k} \to 0 \),
\( \Omega_{\varepsilon_k} \) converges to some smooth \( \Omega \) local graphically and with multiplicity one as \( k\to \infty \).
Up to a subsequence, the locally graphical convergence with multiplicity one implies that the pointed \( (\Sigma_{\varepsilon_{k}}, p_{\varepsilon_{k}}) \) converges to \( (\Sigma:=\partial \Omega, p) \) in the pointed smooth topology. 
From the Harnack inequality for \( v_{\varepsilon_{k}} \), we conclude that \( v_{\varepsilon_{k}} \) converges smoothly to some positive \( u\) with \( u(p)=1\).
Assume that the limit metric of \( \hat{g}_{\varepsilon_{k}} \) is \( \hat{g} \), then \( \hat{g} = u^{\frac{4\alpha}{n-3}}g|_{\Sigma} \).

From \eqref{eq almost ode} and \eqref{eq lower bound on scalar curvature after conformal change}, we see that \( \operatorname{Sc}_{\partial \Omega_{\varepsilon_{k}}} (\hat {g}_{\varepsilon_{k}}) \geq  -C \varepsilon_{k} \) on the compact set \( K \) for some positive constant \( C \). 
And outside \( K \), we have already shown that \(  \operatorname{Sc}_{\partial \Omega_{\varepsilon_{k}}} (\hat {g}_{\varepsilon_{k}})   \) is non-negative.
Now we can apply \cite[Proposition 3.2]{zhu-rigidity-2020} to obtain that the limit \( (\Sigma, \hat{g}) \) is Ricci flat.

Considering the Ricci flatness of the limit \( (\Sigma, \hat{g}) \), \eqref{eq full conformal change} and \eqref{eq lower bound on scalar curvature after conformal change},
we see that \( u \) is a constant. Hence \( u=1 \) on \( \Sigma \) by the fact that \( u(p)=1 \).

Denote by \( \nu_0 \) the limit of \( \nu_{\varepsilon_{k}} \).
Considering again the limit of \eqref{eq lower bound on scalar curvature after conformal change}, we see that
\( Q_0 =0 \)
(setting \( \varepsilon =0\)). Therefore, we must have \( \langle \nabla \zeta, \nu_0 \rangle =1 \).
Recall that \(  |\nabla \zeta | \leq 1   \), we must have \(  \nabla \zeta= \nu_0  \), which implies that \( \Sigma \) is a level set of \( \zeta \). And we conclude the proof.
\end{proof}

\subsection{Rigidity}

We can construct a foliation near \( \Sigma \) as well.
The proof is similar to Lemma \ref{Lem-foliation-construct}. Then we use the foliation to extend the rigidity of \( \Sigma \) to every leaf. It remains to calculate the metric.

\begin{lemma}\label{lm scalar foliation}
  Let \( \Sigma \) be constructed as in Lemma \ref{scalar stable existence}, then there exists a foliation \( \{ \Sigma_t\}_{t\in (-\varepsilon,\varepsilon)} \) near \( \Sigma \) such that \( H + \gamma u^{-1}u_{\nu} -h \) is constant along \( \Sigma_t \).
  \end{lemma}

\begin{proof} Let
\[ Q = \frac{- n \gamma + \gamma + 2 n}{4 (n - 1) + 2 \gamma (2
   - n)} h^2 + h_{\nu} + \Lambda, \]
and
\[ L = - \tfrac{4}{4 - \gamma} \Delta_{\Sigma} + \tfrac{1}{2}
   \mathrm{Sc}_{\Sigma} - Q  . \]
We have $Q \geq 0$ along $\Sigma$.

Since $\Omega$ is a stable critical point of $E (\Omega)$, \eqref{eq bilinear
energy nonnegative} holds for $\partial \Omega = \Sigma$. It follows then that
the first eigenvalue $\lambda_1$ of $L$ is non-negative. We set $u$ to be the
first eigenfunction, that is $L u = \lambda_1 u$. We consider the metric
$\hat{g} = (u^{\alpha})^{\tfrac{4}{n - 3}} g_{| \Sigma}$ on $\Sigma$, where
$\alpha$ is defined in \eqref{eq def alpha}. By the conformal change formula
for the scalar curvature,
\[ (u^{\alpha})^{\frac{n + 1}{n - 3}} \mathrm{Sc}_{\Sigma} (\hat{g}) =
   u^{\alpha - 1}  \left( \mathrm{Sc}_{\Sigma} u - \tfrac{4 (n - 2)}{n - 3}
   \alpha \Delta u - \tfrac{4 (n - 2)}{n - 3} \alpha (\alpha - 1) u^{- 1} |
   \nabla u|^2 \right) . \]
We see by $0 < \alpha < 1$ that
\[ (u^{\alpha})^{\frac{n + 1}{n - 3}} \mathrm{Sc}_{\Sigma} (\hat{g}) \geq
   u^{\alpha - 1} (L u + Q u) \geq 0. \label{eq g hat non-negativity} \]
However, $\Sigma$ admits a map of non-zero degree to $T^{n - 1}$, hence
$\ensuremath{\operatorname{Sc}}_{\Sigma} (\hat{g}) = 0$. Then $\lambda_1 = 0$,
and $u$ is a constant. Hence
\begin{equation}
  \ensuremath{\operatorname{Sc}}_{\Sigma} (g_{|\Sigma}) = 0. \label{eq inf rig ricci 0}
\end{equation}
Denote by $A^0$ the traceless part of the second fundamental form. Now tracing back the equalities leads to that
\begin{align}
  h_{\nu} & = - | \nabla h|,  \label{eq inf rig ricci 1}\\
  \nabla \zeta & = \nu,  \label{eq inf rig ricci 2} \\
  A^0 & = 0,  \label{eq inf rig ricci 3} \\
Q & =\frac{- n \gamma + \gamma + 2 n}{4 (n - 1) + 2 \gamma (2 - n)} h^2 + h_{\nu}
+ \Lambda =0,  \label{eq inf rig ricci 4} \\
  w_{\nu} = (\log u)_{\nu} & = \tfrac{1}{2 (n - 1) \left( \tfrac{n}{2 (n - 1)}
                         \gamma - \gamma + 1 \right)} h,  \label{eq inf rig ricci 5} \\
  u & \text{ is a constant along } \Sigma. \label{eq inf rig ricci 6} 
\end{align}
The above leads the linearization of $H+\gamma u^{-1} u_{\nu}-h$ to just $-\Delta_{\Sigma}$, and this is sufficient to show the existence of a foliation $\{\Sigma_t \}_{t
  \in (- \varepsilon, \varepsilon)}$ near $\Sigma$ such that each leaf is of constant \( H + \gamma u^{-1} u_{\nu} -h \).
The rest of the proof is similar to the proof of Lemma
\ref{Lem-foliation-construct} and we omit it.
\end{proof}

For convenience, we give the following definition.
\begin{definition} \label{def inf rig ricci}
  We call a hypersurface $\Sigma$ in $M$ infinitesimally rigid if it satisfies
  \[ H = -\gamma u^{-1} u_{\nu} +h ,\]
  \eqref{eq inf rig ricci 0}, \eqref{eq inf rig ricci 1}, \eqref{eq inf rig
  ricci 2}, \eqref{eq inf rig ricci 3}, \eqref{eq inf rig ricci 4}, \eqref{eq
  inf rig ricci 5} and \eqref{eq inf rig ricci 6}.
\end{definition}

  We can determine the sign of \( H + \gamma u^{-1} u_{\nu} -h \) for every \( \Sigma_t \). 
\begin{lemma}\label{lm scalar foliation sign}
  Let  \( \{ \Sigma_t\}_{t\in (-\varepsilon,\varepsilon)} \) be constructed as in Lemma \ref{lm scalar foliation},
  then
  \begin{equation}\label{eq H tilde t def}
    \tilde{H}_t := H+\gamma u^{-1} u_{\nu} -h \leq  0 
  \end{equation}
  for \( t\in [0,\varepsilon) \) and \( \tilde{H}_t \geq 0 \) for \( t \in (-\varepsilon,0) \).
  \end{lemma}

  \begin{proof}
    It suffices to prove for $t\in (0,\varepsilon)$, and it is similar for \( t\in (-\varepsilon,0) \). We do this by establishing an ordinary differential inequality for \( \tilde H \).

Let $\nu_t$ be the unit normal of $\Sigma_t$ pointing to the direction of the
foliation and $\phi_t = \langle \partial_t, \nu_t \rangle$ is then a positive
function on $\Sigma_t$. We set
\begin{equation}
  \phi_t = u^{- \gamma / 2} e^{\xi_t}, w = \log u \label{eq xi def}
\end{equation}
and
\begin{equation}
  P_t = - |A_{\Sigma_t} |^2 -\ensuremath{\operatorname{Ric}} (\nu_t, \nu_t) -
  \gamma w_{\nu}^2 + \gamma u^{- 1} \Delta u - \gamma H w_{\nu} - h_{\nu} .
  \label{eq P def}
\end{equation}
The first variation formula for $\tilde{H}_t$ (see Lemma \ref{weighted
stability}) gives
\[ \phi_t^{- 1} \tilde{H}' = - \phi_t^{- 1} \Delta_{\Sigma_t} \phi_t - \gamma
   u^{- 1} \Delta_{\Sigma_t} u - \gamma \phi^{- 1}_t \langle \nabla_{\Sigma_t}
   w, \nabla_{\Sigma_t} \phi_t \rangle + P_{t}. \]
Using \eqref{eq xi def} in the above and after a tedious calculation, we see
\begin{equation}
  \phi_t^{- 1} \tilde{H}' = - | \nabla_{\Sigma_t} \xi_t |^2 -
  \Delta_{\Sigma_t} \xi_t + (\tfrac{\gamma^2}{4} - \gamma) | \nabla_{\Sigma_t}
  w|^2 - \tfrac{\gamma}{2} \Delta_{\Sigma_t} w + P_t . \label{eq use xi in
  tilde H}
\end{equation}
The estimate below 
\[ 2 (|A_{\Sigma_t} |^2 +\ensuremath{\operatorname{Ric}}(\nu_t, \nu_t))
   \geq \ensuremath{\operatorname{Sc}}_g -\ensuremath{\operatorname{Sc}}_{\Sigma_t} + \tfrac{n}{n - 1}
   (\tilde{H}_t - \gamma w_{\nu_t} + h)^2  \]
follows from \eqref{eq:sy rewrite} and
\eqref{eq H tilde t def}.
Using \eqref{eq P def}, the above, \eqref{eq spectral scalar} in \eqref{eq H
tilde t def}, and with suitable regrouping of terms, we obtain
\begin{align}
  \phi_t^{- 1} \tilde{H}' \leq & - | \nabla_{\Sigma_t} \xi_t |^2 -
  \Delta_{\Sigma_t} \xi_t + (\tfrac{\gamma^2}{4} - \gamma) | \nabla_{\Sigma_t}
  w|^2 - \tfrac{\gamma}{2} \Delta_{\Sigma_t} w \\
  & \quad - \Lambda + \tfrac{1}{2} \ensuremath{\operatorname{Sc}}_{\Sigma_t}
  - \tfrac{n}{2 (n - 1)} (- \gamma w_{\nu} + h)^2 - [\gamma (- \gamma w_{\nu}
  + h)w_{\nu_t} + h_{\nu_t} + \gamma w_{\nu_t}^2] \\
  & \quad - \tfrac{n}{2 (n - 1)} \tilde{H}_t^2 - \tilde{H}_t (\tfrac{n}{n -
  1} (- \gamma w_{\nu_t} + h) + \gamma w_{\nu_t}) \\
  =: & L (t) - \tfrac{n}{2 (n - 1)} \tilde{H}_t^2 - \tilde{H}_t (\tfrac{n}{n -
  1} (- \gamma w_{\nu_t} + h) + \gamma w_{\nu_t}) . \label{eq L t def}
\end{align}
We set $q (t) = \phi_t (\tfrac{n}{n - 1} (- \gamma w_{\nu_t} + h) + \gamma
w_{\nu_t})$. Then
\[\tilde{H}'+q(t)\tilde{H}_{t}\leq \phi_{t}L(t).\]
For any positive smooth function $\varphi$ on $\Sigma_{t}$, we have
\[\tilde{H}'+\tfrac{\int_{\Sigma_{t}}q(t)\varphi}{\int_{\Sigma_{t}}\varphi}\tilde{H}_{t}\leq \tfrac{\int_{\Sigma_{t}}\varphi\phi_{t}L(t)}{\int_{\Sigma_{t}}\varphi},\]
since $\tilde{H}_{t}$ depends only on $t$. Let $\Phi(t)=\frac{\int_{\Sigma_{t}}q(t)\varphi}{\int_{\Sigma_{t}}\varphi}$. Then
\[\left( \tilde{H}_t
   e^{\int^t_0 \Phi(s) \mathrm{d} s} \right)'=(\tilde{H}'+\Phi(t)\tilde{H}_{t})e^{\int^t_0 \Phi(s) \mathrm{d} s}\leq e^{\int^t_0 \Phi(s) \mathrm{d} s}\tfrac{\int_{\Sigma_{t}}\varphi\phi_{t}L(t)}{\int_{\Sigma_{t}}\varphi}.\]

   The proof is concluded if we can show that there exists a positive function $\varphi$ such that $\int_{\Sigma_{t}}\varphi\phi_{t}L(t) \leq 0$ for each $t \in (0,
\varepsilon)$. Indeed, $\tilde{H}_t e^{\int^t_0 \Phi (s) \mathrm{d} s}$ is a
non-increasing function on $[0, \varepsilon)$, so
\[ \tilde{H}_t e^{\int^t_0 \Phi (s) \mathrm{d} s} \leq \tilde{H}_0 = 0
   \text{ for } t \in [0, \varepsilon) . \]
Now  show that there exists a positive function $\varphi$ such that $\int_{\Sigma_{t}}\varphi\phi_{t}L(t) \leq 0$ for each $t \in [0, \varepsilon)$. We
assume otherwise: $\int_{\Sigma_{t}}\varphi\phi_{t}L(t) > 0$ for some $t \in [0, \varepsilon)$ and any $\varphi>0$. Let $\psi^2=\varphi\phi_{t}$, then
\begin{equation}
  \int_{\Sigma_t} L (t) \psi^2 > 0 \label{eq Lt inegration}
\end{equation}
for any non-zero positive smooth function $\psi \in C^{\infty} (\Sigma_t)$.

We analyze only terms involving $\tfrac{\gamma}{2} \Delta_{\Sigma_t} w$ and
$| \nabla_{\Sigma_t} \xi |^2 + \Delta_{\Sigma_t} \xi_t$.
First, we see that
\begin{align}
  & (| \nabla_{\Sigma_t} \xi |^2 + \Delta_{\Sigma_t} \xi_t) \psi^2
  \\
  = & | \nabla_{\Sigma_t} \xi |^2 \psi^2 - 2 \langle \nabla_{\Sigma_t} \psi,
  \psi \nabla_{\Sigma_t} \xi \rangle
  +\ensuremath{\operatorname{div}}_{\Sigma_t} (\psi^2 \nabla_{\Sigma_t} \xi_t)
  \\
  \geq & - | \nabla_{\Sigma_t} \psi |^2
  +\ensuremath{\operatorname{div}}_{\Sigma_t} (\psi^2 \nabla_{\Sigma_t} \xi_t)
  . 
\end{align}
So it follows from integration by parts that
\begin{equation}
  - \int_{\Sigma_t} (| \nabla_{\Sigma_t} \xi |^2 + \Delta_{\Sigma_t} \xi_t)
  \psi^2 \leq \int_{\Sigma_t} | \nabla_{\Sigma_t} \psi |^2 . \label{eq
  ibp 1}
\end{equation}
And also
\begin{equation}
  \tfrac{\gamma}{2} \int_{\Sigma_t} \psi^2 \Delta_{\Sigma_t} w =-
  \int_{\Sigma_t} \gamma \psi \langle \nabla_{\Sigma_t} w , \nabla_{\Sigma_t}
  \psi \rangle . \label{eq ibp 2}
\end{equation}
Putting the above two formulas into \eqref{eq Lt inegration}, we see
\begin{align}
  0 < & \int_{\Sigma_t} | \nabla_{\Sigma_t} \psi |^2 + \int_{\Sigma_t} [\gamma
  \psi \langle \nabla_{\Sigma_t} w, \nabla_{\Sigma_t} \psi \rangle +
  (\tfrac{\gamma^2}{4} - \gamma) \psi^2 | \nabla_{\Sigma_t} w|^2] \\
  & \quad + \int_{\Sigma_t} [\tfrac{1}{2} \mathrm{Sc}_{\Sigma_t} -
  \tfrac{n}{2 (n - 1)} (- \gamma w_{\nu} + h)^2 - \Lambda] \psi^2 \\
  & \quad - \int_{\Sigma_t} [\gamma (- \gamma w_{\nu} + h) w_{\nu} + h_{\nu}
  + \gamma w_{\nu}^2] \psi^2 . 
\end{align}
This is the same form as the inequality that appeared in the Theorem
\ref{Thm-scalar} (see the lines near \eqref{eq similar back ref}). Note
also that the inequality is strict.

Now we argue similarly as in Theorem \ref{Thm-scalar}.
The inequality would imply that $\Sigma_t$ admits a metric of positive scalar
curvature. However, this is impossible because $\Sigma_t$ admits a map of
non-zero degree to $T^{n - 1}$. Here, the proof is complete.
\end{proof}


Now we give the proof of Theorem \ref{thm rigid scalar}.

\begin{proof}[Proof of Theorem \ref{thm rigid scalar}]
  
  First, we define \( \Omega_t \) by setting
 \begin{eqnarray}
\Omega_{t}=\begin{cases}
    \Omega \cup \tilde{\Omega}_{t},\mbox{ if } 0<t< \epsilon,\\
    \Omega \setminus \tilde{\Omega}_{t},\mbox{ if } -\epsilon <t<0.
\end{cases}    
\end{eqnarray} 
from the foliation \( \{ \Sigma_t\}_{t\in (\epsilon,\epsilon)} \) constructed in Lemma \ref{lm scalar foliation}.
Here $\tilde{\Omega}_{t}$ is the region bounded by $\Sigma_{t}$ and $\Sigma$. 
We show that every $\Omega_t$ is a stable critical point of the functional $E $. It is
enough to show for $t > 0$, and it is similar for $t < 0$.

Let \( \partial_t \) be the variational vector field of the foliation and 
$\phi_t = \langle \nu_t, \partial_t \rangle$, by the first
variation of $E $ (Lemma \ref{lem-first-variation}),
\[ \tfrac{\mathrm{d}}{\mathrm{d} t} E (\Omega_t) = \int_{\Sigma_t} \tilde{H}_t
   u^{\gamma} \phi_t \leq 0. \]
where we used $\tilde{H}_t \leq 0$ from Lemma \ref{lm scalar foliation
sign}. Therefore, $E (\Omega_t) \leq E (\Omega_0)$ and every $\Omega_t$
is a minimiser to the functional $E$.

Hence, every $\Sigma_t$ is infinitesimally rigid. We use the definition of
infinitesimal rigidity to finish the proof. By umbilicity, we find that the
metric $g$ is a warped product $\mathrm{d} t^2 + \xi (t)^2 g_{T^{n
    - 1}}$ for some flat metric on \( T^{n-1} \),
so the mean curvature of the $t$-level set is given by
\[ H = (n - 1) \tfrac{\xi'}{\xi} . \]
By the fact that $\nabla \phi = \nu$, we see that $\phi$ differs from $t$ by some
constant. We can solve the ordinary differential equation \eqref{eq inf rig
ricci 4} (note that $h_{\nu} = h'$) and obtain
\[ h (t) = - \tfrac{\sqrt{\Lambda}}{\sqrt{\frac{- n \gamma + \gamma + 2 n}{4 (n
   - 1) + 2 \gamma (2 - n)} }} \tan \left( \sqrt{\tfrac{\Lambda (- n \gamma +
   \gamma + 2 n)}{4 (n - 1) + 2 \gamma (2 - n)} } (t + c) \right) \]
in terms of $t$ for some constant $c$. Using \eqref{eq inf rig ricci 5}, we
find an expression for $(\log u)_{\nu}$ and finally an expression for the mean
curvature
\[ H = - \gamma u^{- 1} u_{\nu} + h = \tfrac{(2 - \gamma) (n - 1)}{2 (n - 1) +
   \gamma (2 - n)} h = (n - 1) \xi' / \xi, \]
so
\[ \xi' \xi^{- 1} = \tfrac{2 - \gamma}{\gamma (2 - n) + 2 (n - 1)} h. \]
We easily find that
\[ \xi (t) = \left(\cos (\sqrt{\tfrac{\Lambda (- n \gamma + \gamma + 2 n)}{4 (n -
   1) + 2 \gamma (2 - n)} } (t + c))\right)^{ 2 (2 - \gamma) (-n \gamma + \gamma + 2 n)^{- 1}}. \]
By connectedness, we can extend the rigidity to all of $M$, and by considering the boundary values of $\zeta$ (i.e. $t$), we see that $c = 0$.
Hence the theorem is proved.
\end{proof}




\appendix

\section{Determine the metric in the spectral Ricci case}\label{subsec-determine-the-metric}
 
Here, we give the details of how to determine the metric in Theorem \ref{ricci rigid reformulation}. As shown earlier in Section \ref{sec spectral ricci}, the metric \( g \)
is a doubly warped product
\begin{equation}
  \mathrm{d} t^2 + \phi (t)^2 \mathrm{d} s_1^2 + \varphi (t)^2 \mathrm{d}
  s_2^2 \label{eq doubly warped} .
\end{equation}
Here, to avoid using subscripts, we have set \( \phi = \phi_1 \) and \( \varphi = \phi_2 \).
The function \( h = h(t) \) satisfies the ODE
\begin{equation}
  (1 - \gamma / 4) h^2 + h' +  \Lambda = 0. \label{eq h ode in spectral ricci}
\end{equation}
The function \( u \) is constant on each \( t \)-level set, 
\begin{equation}
 w'= (\log u)' = \frac{1}{2}h, \label{eq w and h}
\end{equation}
and the mean curvature \( H(t) \)
of the level set is given by
\begin{equation}
H = -\gamma u^{-1} u' + h = (-\frac{1}{2} \gamma + 1)h. \label{eq H h}
\end{equation}
Let \( \{e_1, e_2\} \) be an arbitrary orthonormal frame of the \( t \)-level set, then
\begin{equation} \operatorname{Ric}(e_1,e_1) = \operatorname{Ric}(e_2,e_2) \leq \operatorname{Ric}(\partial_t, \partial_t). \label{eq ricci relation}    \end{equation}

We put our calculation in the form of a lemma. 

\begin{lemma} \label{lm calc spectral ricci}
  Let \( g \) be a metric of the form \eqref{eq doubly warped} which satisfies
  \eqref{eq h ode in spectral ricci}, \eqref{eq w and h}, \eqref{eq H h} and \eqref{eq ricci relation}, then
\begin{align}
    \phi(t)&=\phi(0)\left(\cos(\sqrt{\Lambda(1-\tfrac{\gamma}{4})}t)\right)^{\tfrac{2-\gamma}{4-\gamma}} \exp\left( { \tfrac {\phi'(0)}{\phi(0)}   \int_{0}^{t} \left(\cos(\sqrt{\Lambda(1-\tfrac{\gamma}{4})}s)\right)^{-\tfrac{1-\gamma/2}{1-\gamma/4}}ds} \right) \label{eq phi}\\
\varphi(t)&=\varphi(0)\left(\cos(\sqrt{\Lambda(1-\tfrac{\gamma}{4})}t)\right)^{\tfrac{2-\gamma}{4-\gamma}}\exp\left({  \tfrac {\varphi'(0)}{\varphi(0)}   \int_{0}^{t} \left(\cos(\sqrt{\Lambda(1-\tfrac{\gamma}{4})}s)\right)^{-\tfrac{1-\gamma/2}{1-\gamma/4}}ds}\right) \label{eq varphi},
\end{align}
where \( \tfrac{\phi'(0) } {\phi(0)} =- \tfrac{\varphi'(0) } {\varphi(0)} \) and
\begin{equation}
  \tfrac{1}{2} (1 - \tfrac{\gamma}{2}) \Lambda \geq 2
  \left(\tfrac{\phi'(0)}{\phi (0)} \right)^2 . \label{eq beta restriction}
  \end{equation}

\end{lemma}

We find it convenient to record 
the non-zero components of the Ricci curvatures of the metric \( g \) given in \eqref{eq doubly warped} here:
\begin{align}
  \ensuremath{\operatorname{Ric}} (\partial_t, \partial_t) & = - (\phi
  \varphi)^{- 1} (\phi \varphi'' + \varphi'' \phi),  \label{eq ric t}\\
  \ensuremath{\operatorname{Ric}} ( \phi(t)^{-1} \frac{\partial}{\partial s_1} ,  \phi(t)^{-1} \frac{\partial}{\partial s_1} ) & = - (\phi \varphi)^{- 1} 
  (\varphi \phi'' + \varphi' \phi'),  \label{eq ric e1} \\
  \ensuremath{\operatorname{Ric}} ( \varphi(t)^{-1} \frac{\partial}{\partial s_2},  \varphi(t)^{-1} \frac{\partial}{\partial s_2} ) & = - (\phi \varphi)^{- 1}
  (\varphi'' \phi + \varphi' \phi') .  \label{eq ric e2}
\end{align}
Note that 
\[ \{  \phi(t)^{-1} \frac{\partial}{\partial s_1}, 
  \varphi(t)^{-1} \frac{\partial}{\partial s_2} ,  \partial_t  \} \] is an orthonormal frame with respect to the metric \( g \).

\begin{proof}[Proof of Lemma \ref{lm calc spectral ricci}]
The solution $h = h (t)$ to the ODE \eqref{eq h ode in spectral ricci} is given by
\[ h (t) = - \sqrt{\tfrac{ \Lambda}{1 - \tfrac{1}{4} \gamma}} \tan
   (\sqrt{ \Lambda (1 - \tfrac{\gamma}{4})} t) . \]
 It follows from \eqref{eq w and h} that
\[ u = \cos \left( \sqrt{ \Lambda (1 - \tfrac{\gamma}{4})} t
   \right)^{\tfrac{1}{2 (1 - \gamma / 4)}} . \]
 Choosing \( e_1 =   \phi(t)^{-1} \frac{\partial}{\partial s_1}\) and 
 \(e_2 =  \varphi(t)^{-1} \frac{\partial}{\partial s_2} \) in \eqref{eq ricci relation}, and using \eqref{eq ric e1} and \eqref{eq ric e2} leads to 
\begin{equation}
  \phi'' / \phi = \varphi'' / \varphi . \label{eq:equal ricci}
\end{equation}
The mean curvature of the $t$-slice by \eqref{eq doubly warped} and \eqref{eq H h} is
\begin{equation}
  H = \phi' / \phi + \varphi' / \varphi = (- \tfrac{1}{2} \gamma + 1) h.
  \label{eq:mean}
\end{equation}
By \eqref{eq:equal ricci} and \eqref{eq:mean}, 
\begin{align}\label{eq:difference}
    \left(\tfrac{\phi'}{\phi}-\tfrac{\varphi'}{\varphi}\right)'&=\tfrac{\phi''}{\phi}-\left(\tfrac{\phi'}{\phi}\right)^2-\tfrac{\varphi''}{\varphi}+\left(\tfrac{\varphi'}{\varphi}\right)^2\\
    &=\left(\tfrac{\varphi'}{\varphi}-\tfrac{\phi'}{\phi}\right)\left(\tfrac{\varphi'}{\varphi}+\tfrac{\phi'}{\phi}\right)\\
    &=-\left(\tfrac{\phi'}{\phi}-\tfrac{\varphi'}{\varphi}\right)(-\tfrac{1}{2}\gamma+1)h
\end{align}
Then,
\begin{align}
    \tfrac{\phi'}{\phi}-\tfrac{\varphi'}{\varphi}&=2 \beta\left(\cos(\sqrt{\Lambda(1-\tfrac{\gamma}{4})}t)\right)^{-\tfrac{1-\gamma/2}{1-\gamma/4}}
\end{align}
for some constant \( \beta \).
We have by the above and \eqref{eq:mean} that
\begin{align}
    \tfrac{\phi'}{\phi}&=\tfrac{\phi'(0)}{\phi (0)}\left(\cos(\sqrt{\Lambda(1-\tfrac{\gamma}{4})}t)\right)^{-\tfrac{1-\gamma/2}{1-\gamma/4}}+\tfrac{1}{2}(- \tfrac{1}{2} \gamma + 1) h,\\
    \tfrac{\varphi'}{\varphi}&=\tfrac{\varphi'(0)}{\varphi (0)}\left(\cos(\sqrt{\Lambda(1-\tfrac{\gamma}{4})}t)\right)^{-\tfrac{1-\gamma/2}{1-\gamma/4}}+\tfrac{1}{2}(- \tfrac{1}{2} \gamma + 1) h,\\
 \beta = \tfrac{\phi'(0)}{\phi (0)} &=  -\tfrac{\varphi'(0)}{\varphi(0)}, \\
    \left(\log(\tfrac{\phi}{\varphi})\right)'&=2\tfrac{\phi'(0)}{\phi (0)}\left(\cos(\sqrt{\Lambda(1-\tfrac{\gamma}{4})}t)\right)^{-\tfrac{1-\gamma/2}{1-\gamma/4}}.
\end{align}
Then
\[\tfrac{\phi}{\varphi}=\tfrac{\phi(0)}{\varphi(0)}e^{2 \frac{\phi'(0)}{\phi(0)}  \int_{0}^{t} \left(\cos(\sqrt{\Lambda(1-\tfrac{\gamma}{4})}s)\right)^{-\tfrac{1-\gamma/2}{1-\gamma/4}}ds}.\]
We also have,
\[\phi\varphi=\phi(0)\varphi(0)\left(\cos(\sqrt{\Lambda(1-\tfrac{\gamma}{4})}t)\right)^{\tfrac{1-\gamma/2}{1-\gamma/4}}\]
from \eqref{eq:mean}.
From here, we easily obtain \eqref{eq phi} and \eqref{eq varphi}.
It remains to derive a consequence of the inequality in \eqref{eq ricci relation}. And this would put a restriction on \( \beta = \phi'(0) / \phi(0) \).

By choosing \( e_1 =   \phi(t)^{-1} \frac{\partial}{\partial s_1}\) and 
 \(e_2 =  \varphi(t)^{-1} \frac{\partial}{\partial s_2} \) in \eqref{eq ricci relation} again, and using \eqref{eq ric t}, \eqref{eq ric e1} and \eqref{eq ric e2}, we have
\[ \tfrac{\phi''}{\phi} = \tfrac{\varphi''}{\varphi} \geq \tfrac{\varphi'
   \phi'}{\varphi \phi} . \]
We express $\phi'' / \phi$ in terms of $\phi' / \phi$ and $\varphi' / \varphi$
as follows:
\[ \phi'' / \phi = \varphi'' / \varphi = \tfrac{1}{2} [(\tfrac{\phi'}{\phi} +
   \tfrac{\varphi'}{\varphi})' + (\tfrac{\phi'}{\phi})^2 +
   (\tfrac{\varphi'}{\varphi})^2] . \]
Since $\phi' / \phi + \varphi' / \varphi = 0$ at $t=0$, we can set
\begin{align}
  a  &= \beta \cos \left(\sqrt{\Lambda (1 - \gamma / 4)} t\right)^{- \tfrac{1 - \gamma / 2}{1
  - \gamma / 4}}, \\
  d & = \tfrac{1}{2} (- \tfrac{\gamma}{2} + 1) h, \\
 \phi' / \phi &= a + d, \text{ } \varphi' / \varphi = - a + d.
\end{align}
Then
\begin{align}
  0 \geq & \tfrac{\phi''}{\phi} - \tfrac{\varphi' \phi'}{\varphi \phi} \\
  = & \tfrac{1}{2} [(\tfrac{\phi'}{\phi} + \tfrac{\varphi'}{\varphi})' + (\tfrac{\phi'}{\phi})^2 + (\tfrac{\varphi'}{\varphi})^2] - \tfrac{\varphi' \phi'}{\varphi \phi} \\
  = & \tfrac{1}{2} (\tfrac{\phi'}{\phi} + \tfrac{\varphi'}{\varphi})' + \tfrac{1}{2} (a + d)^2 + \tfrac{1}{2} (- a + d)^2 - (a + d) (- a + d) \\
   = & \tfrac{1}{2} (- \tfrac{\gamma}{2} + 1) h' + 2a^2 \\
  = & - \tfrac{1}{2} (- \tfrac{\gamma}{2} + 1) \Lambda \cos (\sqrt{\Lambda (1
  - \tfrac{\gamma}{4})} t)^{- 2} \\
  & \quad + 2 \beta^2 \cos (\sqrt{\Lambda (1 - \tfrac{\gamma}{4})} t)^{- 2
  \tfrac{1 - \gamma / 2}{1 - \gamma / 4}}  
\end{align}
which gives the restriction \eqref{eq beta restriction}.
  \end{proof}

\section{Band width estimate with Ricci curvature bound} \label{sec ricci}

We are able to generalize the band width estimate \cite[Theorem 5.1]{zhu2021} to zero and negative Ricci curvature lower bound. See also \cite[Main Theorem C]{hirsch-rigid-2025} for a method using spacetime harmonic functions.

\begin{theorem}\label{Thm-Ricci-case}
  For each $\kappa\in \{-1,0,1\}$, we define \ $\eta = \eta_{\kappa}$ to be the function
  which satisfies
  \[ \eta' + \eta^2 + 4 \kappa = 0, \text{ } \eta' < 0. \]
  If $(M =[- 1, 1]\times T^2 , g)$ is a torical band with Ricci
  curvature bound $\ensuremath{\operatorname{Ric}} \geq 2 \kappa$, 
   and there exist constants $0 < t_- < t_+$ for $\kappa\in \{0,-1\}$ and $0 < t_- < t_+<\tfrac{\pi}{2}$ for $\kappa=1$ such that
  the $H_{\partial_{+}M}  \geq \eta (t_+)$ and $H_{\partial_{-} M} \leq 
  \eta (t_-)$, then
  \[ \ensuremath{\operatorname{width}}_{g}
     (M) \leq t_+ - t_- . \]
Here the mean curvature of $H_{\partial_{-}M}$ is computed about the unit inner normal vector and $H_{\partial_{+}M}$ is computed about the unit outer normal vector. 
  Equality is achieved if and only if $(M, g)$ is isometric to the models
  \( ([t_-,t_+]\times T^2,g) \) where \( g \) is given by
  \begin{align}
    g=
      \begin{cases}
          \mathrm{d} t^2 + \sin^{1 + c} t \cos^{1 - c} t \mathrm{d} s^2_1 + \sin^{1 -
   c} t \cos^{1 + c} t \mathrm{d} s^2_2& \kappa=1; \label{eq metric kappa 1}\\
          \mathrm{d} t^2 +  t^{1+c} \mathrm{d} s^2_1 +  t^{1-c} \mathrm{d} s^2_2& \kappa=0;\\
          \mathrm{d} t^2 + \sinh^{1 + c} t \cosh^{1 - c} t \mathrm{d} s^2_1 +
   \sinh^{1 - c} t \cosh^{1 + c} t \mathrm{d} s^2_2& \kappa=-1,
      \end{cases}
  \end{align}
  where $0\leq c\leq 1$.
\end{theorem}

\begin{remark}
  The metric \eqref{eq metric kappa 1} for \( \kappa=1 \)
  is easily checked to be consistent with \cite[(2.8)]{hirsch-rigid-2025},
  while Zhu \cite{zhu2021} only gave the metric \eqref{eq metric kappa 1} with \( c=0,1 \) for \cite[Theorem 5.1]{zhu2021}.
  We can also obtain a band width estimate if we replace the Ricci curvature by the sectional curvature. 
  \end{remark}
  Now we set up the variational problem tailored to Theorem \ref{Thm-Ricci-case}.
We assume that $M$ lies in a slightly larger manifold and we move $\partial_-
M$ slightly outward and obtain a resulting band $M_1$. Let $\mathcal{C}_1$ to
be the collection of all Cacciopolli sets $\Omega$ such that $\Omega$ contains
a neighborhood of $\partial_- M_1$ and define the functional
\[ E (\Omega) =\mathcal{H}^{n - 1} (\partial^{\ast} \Omega \cap
   \ensuremath{\operatorname{int}}M_1) - \int_{\Omega} h \]
where $\partial^{\ast} \Omega$ denotes the reduced boundary of $\Omega$. Let
$\mathcal{C}$ to be the sub-collection of Cacciopolli sets $\Omega$ such that
$\partial^{\ast} \Omega \backslash \partial_- M$ is a subset in the closure of
$M$. We consider the minimisers of $E$ in the class $\mathcal{C}$. In Gromov's
convention {\cite[Section 1.6.4]{Gromov2023}}, a minimiser \( \Omega \)
of $E$ is called a minimising
$h$-bubble and we call \( \Sigma := \partial^{*} \Omega \backslash \partial_-M_1\) an \( h \)-minimising boundary.
\begin{proposition}
  \label{pp ricci case}Let \( \eta \) the function and
  $(M, g)$ be a band given in Theorem
  \ref{Thm-Ricci-case} and $d$ be a Lipschitz function such that
  \[ | \nabla d| \leq 1 \text{ in } M, \text{ and } d (\partial_{\pm} M)
     = t_{\pm} \text{ along } \partial_{\pm} M. \]
   Assume that $\Sigma$ is an  \( h \)-minimising \( C^{2,\alpha} \)-boundary (\( 0<\alpha<1 \)) 
  and homologous to $\partial_- M$ (see Remark \ref{rk stable cmp minimiser}). Here, $h=\eta \circ d$. Then $(M, g)$ is isometric to the models
  given in Theorem \ref{Thm-Ricci-case}. In particular, \( H_{\partial_{\pm }M} = \eta \circ d \) along \( \partial_{\pm} M \)
and the band width is equal
  to $t_+ - t_-$.
\end{proposition}

\begin{proof}
Some parts of the proof already appeared in \cite{zhu2021}.
Here, we make use of the arguments
already written down for Theorem \ref{ricci rigid reformulation}.
We set $h = \eta \circ d$, since $\Sigma$ is a stable surface of mean
curvature $\eta \circ d$, then $\Sigma$ satisfies the inequality
\eqref{stability for h} with
$\gamma = 0$ and $\Lambda = 2 \kappa$.
By following the proof of Theorem \ref{ricci rigid reformulation},
(note that we do not need Lemma \ref{lm existence minimiser spectral ricci} because we assumed apriori the existence of a stable surface) 
we see that $(M,g)$ is isometric to a doubly warped product in the form of \eqref{eq doubly warped} and \( H_{\partial_{\pm}M} = \eta(t_{\pm}) \). The rest is a calculation similar to Lemma \ref{lm calc spectral ricci}.
The mean curvature of the \( t \)-level set is
\begin{equation}
  H = h(t)= \phi'/\phi + \varphi'/\varphi ,
\end{equation}
which satisfies the ODE
\begin{equation}
   h^2 + h' +  4\kappa = 0. \label{eq h ode in ricci}
  \end{equation}
And let \( \{e_1, e_2\} \) be an arbitrary orthonormal frame of the \( t \)-level set, then
\begin{equation} 2\kappa= \operatorname{Ric}(e_1,e_1) = \operatorname{Ric}(e_2,e_2) \leq \operatorname{Ric}(\partial_t, \partial_t). \label{eq ricci relation non-spec}    \end{equation}
We calculate for \( \kappa =0 \) as an example and leave the calculations for the case \( \kappa=\pm 1 \).
When \( \kappa=0 \), and so
\[  h^2  + h' = 0, \text{ }h' < 0 \]
by \eqref{eq h ode in ricci}. This leads to
\begin{equation}
  h = \tfrac{1}{t+t_0} = \phi' / \phi + \varphi' / \varphi  \label{mean curvature}
\end{equation}
for some constant \( t_{0} \). We set \( t_0 =0 \) and adjust the range of \( t \) later.
Setting  \( e_1 =   \phi(t)^{-1} \frac{\partial}{\partial s_1}\) and 
 \(e_2 =  \varphi(t)^{-1} \frac{\partial}{\partial s_2} \) in $\ensuremath{\operatorname{Ric}} (e_1,e_{1}) =\ensuremath{\operatorname{Ric}}
(e_2,e_{2})=0$ and using \eqref{eq ric e1}, \eqref{eq ric e2} gives
\( \varphi''/\varphi=  \phi''/\phi  \). So
\begin{align}
    \left(\tfrac{\phi'}{\phi}-\tfrac{\varphi'}{\varphi}\right)'&=\tfrac{\phi''}{\phi}-\left(\tfrac{\phi'}{\phi}\right)^2-\tfrac{\varphi''}{\varphi}+\left(\tfrac{\varphi'}{\varphi}\right)^2\\
    &=\left(\tfrac{\varphi'}{\varphi}-\tfrac{\phi'}{\phi}\right)\left(\tfrac{\varphi'}{\varphi}+\tfrac{\phi'}{\phi}\right)\\
    &=-\left(\tfrac{\phi'}{\phi}-\tfrac{\varphi'}{\varphi}\right)h,
\end{align}
then
\begin{equation}
  \tfrac{\phi'}{\phi} - \tfrac{\varphi'}{\varphi} = c \tfrac{1}{t} .
\end{equation}
By symmetry of $\phi$ and $\varphi$, we set $c \geq 0$. So \( \phi' / \phi = \tfrac{1 + c}{2 t}, \varphi' / \varphi = \tfrac{1 - c}{2 t}
\), which gives
\[ \phi = c_1 t^{\tfrac{1 + c}{2}}, \varphi = c_2 t^{\tfrac{1 - c}{2}} \]
for some positive constants \( c_1 \) and \( c_2 \).
We can set \( c_1 =c_2=1 \) by
adjusting the length of the circles represented by  \( s_1 \) and \( s_2 \).
By \eqref{eq ric t}, $\ensuremath{\operatorname{Ric}} (\partial_t,\partial_{t}) = \tfrac{1 -
c^2}{2 t^2}$. Another condition that $\ensuremath{\operatorname{Ric}}
(\partial_t,\partial_{t}) \geq 0$ gives $c \leq 1$. So the metric takes
\[ \mathrm{d} t^2 + t^{1 + c} \mathrm{d} s^2_1 + t^{1 - c} \mathrm{d} s^2_2 .
\]
We see from the mean curvatures \( H_{\partial_{\pm}M} \) that \( t \in [t_-,t_+] \). In particular, the band width is \( t_+ - t_{-} \).
  \end{proof}
Now we use an argument inspired by \cite{EGM2021,andersson-area-2009} to show that there exists indeed an $h$-minimising boundary which would conclude the proof of Theorem \ref{Thm-Ricci-case}.
\begin{proof}[Proof of Theorem \ref{Thm-Ricci-case}]
 
We define $d = \min \{\ensuremath{\operatorname{dist}}_g (\partial_- M, x) +
t_-, t_+ \}$ and $h = \eta \circ d$.
We prove the band width estimate by contradiction: we assume that
\begin{equation}
  \ensuremath{\operatorname{width}}_{g} (M) > t_+ - t_- . \label{eq bw geq }
\end{equation}
Then
\begin{align}
  H_{\partial_+ M} \geq & \eta (t_+) = h (\partial_+ M), \\
  H_{\partial_- M} \leq & \eta (t_-) = h (\partial_- M) . 
\end{align}
There are several cases to consider:
(1) $H_{\partial_- M} \lneqq
h$ (2) or  \( \partial_-M \) satisfies
$H_{\partial_- M} - h \equiv 0$ but is not stable (3) or \( \partial_- M \) is stable of mean curvature \(h \). 

\text{{\bfseries{Claim}}}: $\partial_- M$ can be perturbed graphically to a
surface $\Sigma_-$ which satisfies
\[ H_{\Sigma_-} < h \text{ along } \Sigma_-  \]
if we are in cases (1) and (2).

In case (1), we can
run the following mean curvature type flow $F (t, \cdot) : S \to M$ with
\[ \partial_t F = - (H - h) \nu \]
starting from $\partial_- M$ for a very short time. Here, $S$ is a manifold
diffeomorphic to $\partial_- M$ serving as a coordinate space. By the strong
maximum principle, for all sufficiently small enough $t$, $H - h < 0$ (see \cite[Lemma 5.2]{andersson-area-2009}
and $h$ is Lipschitz, which is enough for their
argument). It suffices to put $\Sigma_- = F (t, \cdot)$ with a small $t$.

In case (2), when $H_{\partial_- M} - h = 0$ along
$\partial_- M$ but $\partial_-M$ is not stable. Let $u$ be the first eigenfunction of
\[ L = - \Delta_{\partial_{-}M} - (\ensuremath{\operatorname{Ric}}(\nu,\nu) + |A|^2) -
   \nabla_{\nu} h. \]
 We can choose \( u>0 \).
Then $L u < 0$. Let $X$ be a vector field such that $X = u \nu$ along
$\partial_- M$. The first variation of $H - h$ is $\delta_X (H - h) = L u<
0$. Hence, there exists a $\Sigma_-$ such that $H_{\Sigma_-} - h < 0$.

In case (3), when \( \partial_-M \) is stable with mean curvature \( h \).
First note that stability is sufficent for a construction of a local foliation \( \{ \Sigma_t^-\}_{t\in [0,\varepsilon]} \)
of constant \( H-h \)
near \( \partial_-M \) such that \( \Sigma_0 = \partial_- M \).
Let \( \tilde{H}(t) = H_{\Sigma_t} -h \), by the ODE \eqref{eq:H tilde inequality} established in Proposition \ref{Prop-monotonic-formula}, we see that \( \tilde{H}(t) \leq 0 \) for all \( t \in [0,\varepsilon] \). There are two cases as well: there exists some \( t\in [0,\varepsilon] \) such that \( \tilde{H} (t) <0 \); otherwise \( \tilde{H}(t) \equiv 0 \) on \( [0,\varepsilon] \) which implies that \( \Sigma_{\varepsilon} \) is stable (see Remark \ref{rk eigen implies stability}) and enables the foliation to extend beyond \( \Sigma_{\varepsilon} \).
Now we can assume that the foliation \( \{\Sigma_t^{-}\}_{t\in [0,\varepsilon]} \) is maximal in the sense that every leaf is of vanishing \( H-h \), and
if we extend the foliation
beyond \( \Sigma_{\varepsilon} \) (i) either there is a leaf with \( H-h <0 \) or (ii) \( \Sigma_{\varepsilon} \) touches $\partial_{+}M$, from which it follows by the strong maximum principle (see \cite[Proposition 2.4]{andersson-area-2009}) that \( \cup_{t\in [0,\varepsilon]}\Sigma_t^-=M\).
It is worth noting that $\varepsilon$ can be 0.
If 
\( \cup_{t\in [0,\varepsilon]}\Sigma_t^-=M\), we can show that the band width is \( t_+ - t_- \), which is a contradiction;
otherwise, there is a leaf with \( H-h <0 \) in the foliation extended beyond $\Sigma_{\varepsilon}$ and we take this leaf to be \( \Sigma_- \).
To sum up, there exists a \( \Sigma_- \) such that \( H_{\Sigma_-} -h <0 \).

We can argue similarly that there exists a perturbation \( \Sigma_+ \)
of $\partial_+ M$ such
that $H_{\Sigma_+} - h > 0$. (Observe that it is convenient to reverse the choice of normal, the sign of the mean curvature. And the roles of $\partial_{\pm}M$ are switched.)

Note that any two leaves from the maximal foliations $\{\Sigma_{t}^+\}$ and $\{\Sigma_{t}^-\}$ starting respectively from $\partial_+ M$ and $\partial_- M$ (we have omitted subscripts) cannot touch by the strong maximum principle.

Because of $\Sigma_{\pm}$, we can apply Theorem \ref{lm existence mu bubble}, and there exists an $h$-minimising
boundary $\Sigma$. (We remark that $\Sigma$ is $C^{2,\alpha}$ by \cite[Theorem 4.4]{duzaar-existence-1993} and sufficient for our purpose.) By Proposition \ref{pp ricci case}, \( H_{\Sigma_{\pm}} - h =0 \) along \( \Sigma_{\pm } \) which contradicts the construction of \( \Sigma_{\pm } \).
So we conclude the proof of \( \operatorname{width}_{g}(M) \le t_+ -t_- \).

Now we turn to the rigidity case. If $\ensuremath{\operatorname{width}}_{g} (M)
= t_+ - t_-$, then by the above proof,
there is only one possibility such that \( \partial_-M  \) is stable with \( H_{\partial_-M} -h =0 \),
each leaf of the foliation starting from
\( \partial_-M \) has vanishing \( H-h \) and the foliation covers the whole of \( M \).
And this finishes the proof.
 \end{proof}
 \begin{remark}\label{rk eigen implies stability}
   In case (3), we have used the fact: if \( \{\Sigma_{t}\}_{t \in [0,\varepsilon]} \) is a foliation such that every leaf is of vanishing \( H - h  \), then every leaf \( \Sigma_t \) is stable. This can be see from the first variation of the \( H-h \) is vanishing. So
   \[ L_{\Sigma_t} u_t := -\Delta_{\Sigma_{t}} u_t-[ (\operatorname{Ric} (\nu_{t},\nu_{t}) +|A_{\Sigma_t}|^2) +\nabla_{\nu_t} h]u_t =0\]
   for \( u_t = \langle \partial_t, \nu_t \rangle >0 \) meaning that the first eigenvalue of \( L_{\Sigma_t} \)
   is zero, hence every \( \Sigma_t \) satisfies the stability.
\end{remark}

\bibliographystyle{alpha}
\bibliography{spectral-ricci}
\end{document}